\documentclass[12pt]{article}

\usepackage{latexsym}
\usepackage{a4wide}
\usepackage{amscd}
\usepackage{graphics}
\usepackage{amsmath}
\usepackage{amssymb}
\usepackage{esint}
\usepackage{mathrsfs}
\usepackage{amsthm}
\usepackage{bbm}
\usepackage{color}
\usepackage{accents}
\usepackage{enumerate}
\input xy
\xyoption{all}

\newcommand{\N}{\mathbb{N}}                     % the natural numbers
\newcommand{\R}{\mathbb{R}}                     % the real line
\newcommand{\set}[2]{\left\{{#1}\mid{#2}\right\}}       % the set
\newcommand{\re}{\mathrm{Re\,}}                 % real part
\newcommand{\ind}{\mathrm{ind}}               % Fredholm index
\newcommand{\supp}{\mathrm{supp\,}}             % support
\newcommand{\graf}{\mathrm{graph\,}}            % graph
\newtheorem{mainthm}{\sc Theorem}           % numbered absolutely
\newtheorem{thm}{\sc Theorem}[section]               % numbered absolutely
\newtheorem*{thm*}{\sc Theorem}               % no number
\newtheorem{cor}[thm]{\sc Corollary}        % numbered along with Theorem
\newtheorem*{cor*}{\sc Corollary}        % no number
\newtheorem{lem}[thm]{\sc Lemma}            % numbered along with Theorem
\newtheorem{prop}[thm]{\sc Proposition}     % numbered along with Theorem
\newtheorem{defn}[thm]{\sc Definition}      % numbered along with Theorem
\newtheorem{rem}[thm]{\sc Remark}           % numbered along with Theorem  
\newtheorem{ex}[thm]{\sc Example}           % numbered along with Theorem
\title{Stable foliations and CW-structure induced by a Morse-Smale gradient-like flow}

\author{Alberto Abbondandolo and Pietro Majer}

\date{}

\begin{document}

\maketitle

\begin{abstract}
We prove that a Morse-Smale gradient-like flow on a closed manifold has a ``system of compatible invariant stable foliations'' that is analogous to the object introduced by Palis and Smale in their proof of the structural stability of Morse-Smale diffeomorphisms and flows, but with finer regularity and geometric properties. We show how these invariant foliations can be used in order to give a self-contained proof of the well-known but quite delicate theorem stating that the unstable manifolds of a Morse-Smale gradient-like flow on a closed manifold $M$ are the open cells of a $CW$-decomposition of $M$.
\end{abstract}

\section*{Introduction}

\paragraph{The setting.} Let $M$ be a closed smooth manifold and $\phi: \R \times M \rightarrow M$ a flow of class $C^1$.  We recall that $\phi$ is said to be {\em gradient-like}  if its chain recurrent set consists of stationary points or, equivalently, if there is a continuous real function on $M$ that strictly decreases on each non-stationary orbit of $\phi$. The flow $\phi$ is said to be {\em Morse} if all its stationary points are hyperbolic, meaning that for such a point $x$ the linear mapping $L_x: T_x M \rightarrow T_x M$ which is defined by
\[
d\phi^t(x) = e^{tL_x},
\]
does not have purely imaginary eigenvalues. In this case, the tangent space of $M$ at $x$ has an $L_x$-invariant splitting
\[
T_x M = E_x^u \oplus E^s_x
\]
such that all the eigenvalue of $L_x|_{E^u_x}$ (resp.\ $L_x|_{E^s_x}$) have positive (resp.\ negative) real part. When $\phi$ is gradient-like and Morse, the {\em unstable} and {\em stable manifold} of each stationary point $x$, that is the sets
\[
W^u(x) := \{ p \in M \mid \phi^t(p) \rightarrow x \mbox{ for } t\rightarrow -\infty\}, \; W^s(x) := \{ p \in M \mid \phi^t(p) \rightarrow x \mbox{ for } t\rightarrow +\infty\},
\]
are embedded $C^1$ submanifolds and their tangent spaces at $x$ are $E^u_x$ and $E^s_x$, respectively. The number
\[
\mathrm{ind} (x) := \dim E^u_x = \dim W^u(x)
\]
is called {\em Morse index} of $x$. The flow $\phi$ is said to be {\em Smale} if for every pair of stationary points $x,y$ the unstable manifold of $x$ and the stable manifold of $y$ meet transversally. In particular, they can have a non-empty intersection only if $\mathrm{ind}(x) > \mathrm{ind}(y)$ or $x=y$.

Throughout this paper, $\phi$ will be a  gradient-like Morse-Smale flow of class $C^1$ on the closed manifold $M$. An important example is the negative gradient flow of a $C^2$ Morse function $f: M \rightarrow \R$ with respect to a generic Riemannian metric. Notice that, in this case, the stationary points of $\phi$ are the critical points of $f$ and if $x$ is such a critical point then $L_x$ is minus the Hessian of $f$ at $x$, so its eigenvalues are real.
 
\paragraph{The main result.} It is well known that the unstable manifold of a stationary point of $\phi$ of index $k$ is the image of an open $k$-dimensional ball by a $C^1$ embedding. Moreover, the unstable manifolds of stationary points form a partition of $M$ and the Morse-Smale assumption implies that the closure of each unstable manifold $W^u(x)$ is the union of $W^u(x)$ and some unstable manifolds of index strictly less than $\mathrm{ind}(x)$. This suggests that the unstable manifolds of stationary points should form a $CW$-decomposition of the manifold $M$. What one has to prove in order to have this is the existence of a homeomorphism from an open ball of dimension $\mathrm{ind}(x)$ onto $W^u(x)$ that extends continuously to the closure of the ball. Discussing a proof of this result is the first aim of this paper. Morse precisely, we shall give a detailed proof of the following theorem, in which $B^k$ denotes the open unit ball of $\R^k$.

\begin{mainthm}
\label{introthm1}
Let $\phi$ be a gradient-like Morse-Smale flow of class $C^1$ on the closed manifold $M$. For every stationary point $x$ of $\phi$ with Morse index $k$ there exists a continuous map
\[
\varphi: \overline{B^k} \rightarrow \overline{W^u(x)}
\]
whose restriction to $B^k$ is a homeomorphism onto $W^u(x)$. In particular, the unstable manifolds $W^u(x)$ are the open cells of a $CW$-decomposition of $M$.
\end{mainthm}

Versions of this theorem under various additional assumptions have been proven by several authors. To the best of our knowledge, the above version for general gradient-like Morse-Smale flows is not covered by the results in the literature. Before explaining the idea of our proof, we briefly recall the history of this theorem.

\paragraph{Some history.} In an influential note from 1949, Thom observed that the unstable manifolds of the negative gradient flow of a Morse function are open cells, see \cite{tho49}. He did not assume the Smale transversality condition - which was yet to come - but observed that the cell decomposition that one gets is still good enough to deduce the Morse inequalities and vanishing results for the homotopy groups of manifolds admitting Morse functions without critical point in certain index ranges. He did not address the question of the structure of the closure of unstable manifolds. In this respect, it should be remarked that, without the Smale transversality condition, even the first part of Theorem \ref{introthm1} does not hold: Not only the closure of the unstable manifold of a stationary point $x$ of Morse index $k$ might contain points belonging to other unstable manifolds of dimension $k$ or higher, as well known examples show, but it is actually possible that no homeomorphism from $B^k$ onto $W^u(x)$ extends continuosly to the closure of $B^k$. Since we were unable to find an example in the literature, we include one in Appendix \ref{nonsmale}.

Smale introduced the transversality condition that is now named after him in \cite{sma61}, in which he proved that the negative gradient flow of a given Morse function does satisfy it for a generic choice of the metric. He also extensively studied the properties of what are now called Morse-Smale flows - both with the gradient-like assumption that is considered here and in the more general case in which hyperbolic periodic orbits are allowed - in \cite{sma60}. Dynamical ideas are also pervasive in Milnor's books on Morse theory \cite{mil63} and - even more - on the $h$-cobordism theorem \cite{mil65b}. However, neither Smale nor Milnor addressed the question of the structure of the closure of unstable manifolds directly. But as Franks observed in \cite{fra79}[Theorem 3.2], from Milnor's work it is possible to extract the following statement: Given a Morse function $f$ on $M$, there exists a $CW$-complex that is homotopy equivalent to $M$ and has one $k$-cell for each critical point of $f$ of Morse index $k$.  

The first proof of the fact that unstable manifolds are indeed the open cells of a $CW$-decomposition of $M$  is due to Kalmbach \cite{kal75}. She considered the gradient flow of an arbitrary Morse function $f$ and, following Milnor, chose the metric on $M$ so that - among other conditions - the negative gradient of $f$ has the special form
\begin{equation}
\label{specform}
x_1 \,\frac{\partial}{\partial x_1} + \dots + x_k \,\frac{\partial}{\partial x_k}  - x_{k+1} \,\frac{\partial}{\partial x_{k+1}}  - \dots - x_n \, \frac{\partial}{\partial x_n} 
\end{equation}
in suitable local coordinates near each critical point of $f$. Here $k$ is the Morse index of the critical point. The Morse Lemma guarantees that metrics such that this holds always exist. This is a very natural choice to make when the object of study is the Morse function,  and the same assumption has been considered in most of the subsequent works on this subject. The advantage of the form (\ref{specform}) is not only that the vector field is linear in the special coordinates (recall that the flow near a hyperbolic stationary point can always be linearized, but by a conjugacy that is in general no more than H\"older continuous), but also that the rates of contraction and expansion are the same in all the stable and unstable directions. Condition (\ref{specform}) appears also in the investigations on other aspects of Morse-Smale gradient flows, such as Harvey and Lawson's approach to Morse theory via currents \cite{hl01}, in which Riemannian metrics such that the negative gradient vector field of the Morse function $f$ has the form (\ref{specform}) near each critical point are called $f$-tame (see also \cite{min15} for a way to remove the $f$-tameness condition in this context).

By using the same assumption (\ref{specform}), Laudenbach studied the structure of the closure of the unstable manifolds of gradient flows and showed that they are ``submanifolds with conical singularities'', see \cite{lau92}. From this fact, it is possible to show that the unstable manifolds of a Morse-Smale gradient flow satisfying (\ref{specform}) are the cells of a $CW$-decomposition (see \cite{lau92}[Remark 3]). Two years later, Latour introduced an abstract compactification of the unstable manifolds of a gradient flow that has been widely used since then: Given two critical points $x,y$, once considers the space $\overline{\mathcal{M}}(x,y)$ of all unparametrized broken orbits from $x$ to $y$, that is, $m$-uples $(u_1,\dots,u_m)$ where $u_j$ is an orbit connecting two distinct critical points $x_{j-1}$ and $x_j$, where $x_0=x$ and $x_m=y$. One can put a natural topology on the disjoint union
\[
\widehat{W}^u(x) := W^u(x) \sqcup \bigcup_y \overline{\mathcal{M}}(x,y) \times W^u(y),
\]
where the last union ranges on the set of critical points $y$ of index less than $\mathrm{ind}(x)$, so that this space is compact and the map
\[
\mathrm{ev}: \widehat{W}^u(x) \rightarrow \overline{W^u(x)}
\]
which is the identity on $W^u(x)$ and maps each $(u,p)\in \overline{\mathcal{M}}(x,y) \times W^u(y)$ into $p$ is continuous. A possible strategy leading to Theorem \ref{introthm1} is then to prove that the abstract compactification $\widehat{W}^u(x)$ is homeomorphic to a closed ball. This is not done in Latour's paper, which instead deals with the structure of manifolds with corners of these compactifications, in the more general setting of Morse-Novikov theory.

The approach via abstract compactification has been explicitly worked out  by Audin and Damian in their introductory book on Floer homology \cite{ad14}: Assuming the normal form (\ref{specform}), this book contains a neat and complete proof of the fact that the compactification $\widehat{W}^u(x)$ is homeomorphic to a closed ball. Also very readable and extremely detailed is Qin's paper \cite{qin10}. This paper contains a detailed description of the structure of manifold with corners on $\widehat{W}^u(x)$, as well as a proof of the fact that this space is homeomorphic to a closed ball. Qin works in the more general setting of functionals on possibly infinite dimensional Hilbert manifolds, but his paper is a good reference also for the readers interested in the finite dimensional situation.

A detailed description of the structure of manifold with corners of the abstract compactification of $W^u(x)$ is also contained in Burghelea, Friedlander and Kappeler's paper \cite{bfk10}, which however does not deal with the question of determining the topology of $\widehat{W}^u(x)$. 

In all the above mentioned references, assumption (\ref{specform}) plays an important role. In \cite{bc07} Barraud and Cornea proved Theorem \ref{introthm1} for Morse-Smale gradient flows without assuming the normal form (\ref{specform}). They however assumed that the Morse function has a unique critical point of Morse index zero. Their proof is still based on an abstract compactification and uses gluing techniques from Floer theory as well as the notion of topological transversality. It is somehow less elementary than the above mentioned proofs by Kalmbach, Audin-Damian and Qin.  

Finally, in the subsequent paper \cite{qin11}, Qin showed how Theorem \ref{introthm1} for an arbitrary gradient flow of a Morse function satisfying the Smale transversality condition can be deduced from the corresponding statement for a  Morse-Smale gradient flow having the form (\ref{specform}) near the critical points. The idea, which was actually suggested by Franks in the already mentioned \cite{fra79}, is to show that the given Morse-Smale gradient flow can be connected to a Morse-Smale gradient flow satisfying (\ref{specform}) by a smooth path of gradient flows which are all Morse-Smale and have the same stationary points. The existence of a path with this properties is not immediate, but similar results were proved by Newhouse and Peixoto in \cite{np76}. Once the existence of this path has been established, the structural stability theorem for Morse-Smale flows (see \cite{ps70}) implies that any two gradient flows in the path are topologically equivalent, meaning that there is a homeomorphism of $M$ mapping orbits of the first flow into orbits of the second one. In particular, the two gradient flows at the end-points of the path are topologically equivalent, and hence the abstract compactifications of their unstable manifolds are homeomorphic. 

Summarizing: From the available literature it is possible to deduce Theorem \ref{introthm1} for the negative gradient flow of an arbitrary  smooth Morse function, with respect to any Riemannian metric for which the flow has the Smale transversality condition: One needs to put together a proof for the special form (\ref{specform}) - such as \cite{kal75}, \cite{lau92}, \cite{qin10} or \cite{ad14} - with the argument \`a la Peixoto-Newhouse that is fully explained in \cite{qin11} and the structural stability of Morse-Smale flows from \cite{ps70}. 

\paragraph{A proof based on stable foliations.} The proof of Theorem \ref{introthm1} that is contained in this paper is somehow more direct than the existing ones, as it does not use the abstract compactification. Moreover, it does not need the vector field generating the flow to have the special form (\ref{specform}) near the stationary points, which are just assumed to be hyperbolic. Non real eigenvalues of the linearization of the flow at stationary points are also allowed, so the flow $\phi$ is not necessarily the gradient flow of a Morse function. 

We do however need some high regularity on the flow, and in order not to be forced to keep track of the regularity degrees, we shall just assume the Morse-Smale gradient-like flow $\phi$ to be smooth (see Remark \ref{smooth?} below for comments on this). Theorem \ref{introthm1} in the smooth category immediately implies it in the $C^1$ category: If $\phi$ is a $C^1$ Morse-Smale gradient-like flow $\phi$ then by the structural stability theorem any $C^1$ flow that is $C^1_{\mathrm{loc}}$-close to it is topologically equivalent to it, so it is enough to approximate $\phi$ in the $C^1_{\mathrm{loc}}$ topology with a smooth flow and get the desired conclusion.

Let $x$ be a stationary point of Morse index $k$. The main part of the proof consists in constructing a continuous flow $\psi$ on $W^u(x)$ with the following properties:
\begin{enumerate}[(a)]
\item The only stationary point of $\psi$ is $x$ and $\psi^t(p)=\phi^t(p)$ for every $p$ in a neighborhood of $x$ and every $t\leq 0$. 
\item The orbit $\psi^t(p)$ of every $p\in W^u(x)$ tends to $x$ for $t\rightarrow -\infty$.
\item The orbit $\psi^t(p)$ of every $p\in W^u(x)\setminus \{x\}$ tends to some point $\omega(p)$ in $\overline{W^u(x)} \setminus W^u(x)$ for $t\rightarrow +\infty$.
\item The map $\omega: W^u(x) \setminus \{x\} \rightarrow \overline{W^u(x)}$ is continuous.
\end{enumerate}
Having such a flow, it is immediate to construct a map $\varphi: \overline{B^k} \rightarrow \overline{W^u(x)}$ satisfying the requirements of Theorem \ref{introthm1} (see Corollary \ref{daldisco} below). Note that the restriction of the flow $\phi$ to $W^u(x)$ satisfies (a), (b) and (c) but not (d), except in the trivial cases in which either $x$ has index $0$ (so that $W^u(x)=\{x\}$) or
$\overline{W^u(x)} \setminus W^u(x)$ consists of just one point (which is then a stationary point of index zero).

The main tools in our construction of the flow $\psi$ are invariant stable foliations. Let $x$ be stationary point of $\phi$ of index $k$ and let $U_x$ be an open $\phi$-invariant neighborhood of $x$. By invariance, $U_x$ is actually a neighborhood of $W^u(x) \cup W^s(x)$. An invariant stable foliation of $U_x$ is a foliation
\[
\mathscr{L}^s_x = \{L^s_x(p)\}_{p\in W^u(x)} 
\]
of $U_x$ such that each leaf $L^s_x(p)$ is a smoothly embedded $(n-k)$-dimensional manifold that meets $W^u(x)$ transversally at $p$, such that the invariance condition
\[
\phi^t(L^s_x(p)) = L^s_x(\phi^t(p))
\]
holds for every $p\in W^u(x)$ and every $t\in \R$. In particular, the leaf $L^s_x(x)$ coincides with the stable manifold $W^s(x)$. Similar objects appear in \cite{ps70} under the name of ``tubular families of $W^s(x)$''.

The leaves of this foliation are smooth submanifolds, but the foliation needs not be smooth in the usual sense. It is however ``partially smooth'', meaning the map that to each $p\in W^u(x)$ associates the embedding whose image is $L^s_x(p)$ can be chosen to be continuous into space of smooth embeddings, which is endowed with the $C^{\infty}$ topology. See Definition \ref{admissible} below for a more precise local definition. In particular, the map that maps every point $p\in U_x$ to the tangent space at $p$ to the leaf of $\mathscr{L}^s_x$ through $p$ is continuous into the Grassmannian of $(n-k)$-dimensional subspaces of $TM$. 

Invariant stable foliations are easy to build. More delicate is to show that it is possible to build a system of invariant stable foliations $\mathscr{L}^s_x$ for every stationary point $x$ of $\phi$ such that the foliation $\mathscr{L}^s_x$ refines $\mathscr{L}^s_y$ whenever $\mathrm{ind}(x) \geq \mathrm{ind}(y)$. This means that each leaf of $\mathscr{L}^s_x$ is either disjoint from the domain of $\mathscr{L}^s_y$ or fully contained in a leaf of $\mathscr{L}^s_y$. More precisely, we shall prove the following result.

\begin{mainthm}
\label{introthm2}
Let $\phi$ be a smooth gradient-like Morse-Smale flow on the closed manifold $M$. For every stationary point $x$ of $\phi$ there are an invariant stable foliation $\mathscr{L}^s_x$ of an open invariant neighborhood $U_x$ of $x$ as defined above and an invariant partially smooth foliation $\mathscr{S}^s_x$ of $U_x \setminus W^u(x)$ such that the following conditions hold:
\begin{enumerate}[(i)]
\item If $\mathrm{ind}(x)=k$, then the leaves of $\mathscr{S}^s_x$ are $(n-k-1)$-dimensional spheres and this foliation refines $\mathscr{L}^s_x$. Each leaf of $\mathscr{S}^s_x$ is transverse to every unstable manifold $W^u(y)$.
\item  If $\mathrm{ind}(x) \geq \mathrm{ind}(y)$ then the foliation $\mathscr{L}^s_x$ refines the foliation $\mathscr{S}^s_y$, and a fortiori also $\mathscr{L}^s_y$. Moreover, the foliations $\mathscr{L}^s_x$ and $\mathscr{S}^s_x$ are $\mathscr{L}^s_y$-smooth.
\item Any continuously differentiable curve that is contained in the domain of the foliation $\mathscr{L}^s_x$ or $\mathscr{S}^s_x$ and is tangent to every leaf of this foliation, is fully contained in one leaf.
\end{enumerate}
\end{mainthm}

This result is a refinement of a theorem that Palis and Smale proved in \cite{ps70}[Theorem 2.6] and that is the main ingredient in their proof of the structural stability of Morse-Smale diffeomorphisms and flows. After Palis and Smale's work, different proofs of the structural stability of more general flows were given, such as Robinson's proof of the structural stability of Axiom A flows \cite{rob74}, but systems of compatible stable and unstable foliations similar to the ones appearing in the above result have been often used in the study of Morse-Smale diffeomorphisms and flows (see e.g.\ \cite{pt83} and \cite{bgp17}).

The main additions in our statement are the presence of the further foliations into spheres $\mathscr{S}^s_x$ and a higher regularity of all the foliations. Indeed, Palis and Smale required only smooth leaves with tangent spaces varying continuously. Instead, we are requiring the already mentioned stronger partial smoothness plus the further conditions required in (ii) and (iii), which still need some explanation. Concerning (ii), the precise meaning of $\mathscr{L}'$-smoothness for a foliation $\mathscr{L}$ refining $\mathscr{L}'$ is discussed in Section \ref{admsec} below. Here we just note that it implies in particular that $\mathscr{L}$ restricts to a smooth foliation - in the usual sense - of every leaf of $\mathscr{L}'$. Condition (iii), which would hold automatically for a $C^1$ foliation, does not hold in general for a partially smooth one (see Examples \ref{cubic} and \ref{osculating} below), but turns out to be quite useful. The invariant stable foliations of Theorem \ref{introthm2} might have other applications, for instance to the already mentioned approach  to Morse theory via currents, see \cite{hl01}, and to the spectral analysis of Morse-Smale gradient flows, see \cite{dr19}, which so far has been carried out under the assumption of $C^1$-linearizability near stationary points.

Giving a  complete and self-contained proof of Theorem \ref{introthm2} will take the first five sections of this paper. This proof differs from Palis and Smale's one, as our foliations are obtained by integrating vector fields rather than by constructing retractions. The class of foliations we deal with is discussed in Section \ref{admsec}, and Theorem \ref{introthm2} is proven in Sections \ref{folsec} and \ref{refsec}, building on a general statement about the regularity of the graph transform in hyperbolic dynamics that we prove in Section \ref{gtsec}. The latter statement has some consequences which might have some independent interest, such as the behaviour of the tangent spaces of $W^u(x)$ along sequences of points in $W^u(x)$ converging to some point outside of $W^u(x)$, that are discussed in Section \ref{gensec}.

Building on Theorem \ref{introthm2}, a flow $\psi$ on $W^u(x)$ satisfying the conditions (a)-(d) stated above can be constructed quite easily. Indeed, $W^u(x)\setminus \{x\}$ is covered by the union of all the invariant neighborhoods $U_y$ for $\mathrm{ind}(y) < \mathrm{ind}(x)$. For every such stationary point $y$, the traces of the foliations $\mathscr{L}^s_y$ and $\mathscr{S}^s_y$ on $W^u(x)$ define partially smooth foliations of $U_y \cap W^u(x)$, thanks to the transversality condition stated in Theorem \ref{introthm2} (i). We can define a vector field $Y_y$ on $U_y \cap W^u(x)$ which is tangent to the leaves of $\mathscr{L}^s_y\cap W^u(x)$, tranverse to the leaves $\mathscr{S}^s_y\cap W^u(x)$ and points in the direction in which the radii of the spheres forming the foliation $\mathscr{S}^s_y$ decrease. The flow of the vector field $Y_y$ is positively complete and the orbit of any point in the leaf $L^s_y(p)\cap W^u(x)$ converges to the point $p\in W^u(y)$ for $t\rightarrow +\infty$. 

By patching together the vector fields $Y_y$ using a suitable partition of unity, we can define a vector field $Y$ on $W^u(x)\setminus \{x\}$ whose flow has the property that the orbit of each $p$ near the ``boundary'' $\overline{W^u(x)} \setminus W^u(x)$ converges to some point $\omega(p)$ in $\overline{W^u(x)} \setminus W^u(x)$, where $\omega$ is a continuous map. Here, the partition of unity is chosen in such a way that in the regions in which more vector fields $Y_y$ are present, those corresponding to stationary points of higher Morse index are privileged. The refinement property of the above foliations plays an important role in proving the continuity of $\omega$. See Section \ref{nearbdrysec} for the explicit construction of $Y$ and for the discussion of the properties of its flow. It should also be remarked that the vector field $Y$ is continuous but may not be locally Lipschitz continuous. However, it is uniquely integrable because at every point $p$ it is tangent to the stable leaf of higher dimension containing $p$, on which it is smooth, and thanks to condition (iii) in Theorem \ref{introthm2}. 

The final step is to juxtapose the flow of $Y$, which has the required behaviour near the ``boundary'' of $W^u(x)$, to the restriction of the original flow $\phi$ to $W^u(x)$, which has the right behaviour near $x$. This is done in Section \ref{finalsec}, using some general facts about the juxtaposition of flows that are discussed in Appendix \ref{juxtsec}.

\paragraph{Acknowledgements.} We would like to thank Lizhen Qin for sharing with us his knowledge of the history of this problem  and for precious bibliographical informations. The research of A.\ Abbondandolo is supported by the DFG-Project 380257369 ``Morse theoretical methods in Hamiltonian dynamics''.
   
\tableofcontents   

\section{Partially smooth foliations}
\label{admsec}

In this section $M$ denotes a (not necessarily compact) smooth $n$-dimensional manifold without boundary. Our aim here is to discuss a class of foliations of $M$ that appear often in dynamics. These are not smooth foliations in the usual sense, but they have smooth leaves that vary continuously. Here is the precise definition. 

\begin{defn}
\label{admissible}
A $k$-dimensional partially smooth foliation of $M$ is a partition $\mathscr{L}$ of $M$ such that for every $p\in M$ there exists  a homeomorphism
\[
\varphi: \Omega \rightarrow U
\]
from an open subset $\Omega$ of $\R^k \times \R^{n-k}$ onto an open neighborhood $U$ of $p$ such that
\begin{enumerate}[(i)]
\item $\{ \varphi(\Omega \cap ( \R^k \times \{y\})) \}_{y\in \R^{n-k}} = \{ L \cap U\}_{L\in \mathscr{L}}$;
\item $\varphi$ is infinitely differentiable in the first variable $x\in \R^k$ and all the partial differentials $d_1^h \varphi$, $h\in \N$, are continuous on $\Omega$;
\item for every $(x,y)\in \Omega$ the partial differential $d_1 \varphi(x,y): \R^k \rightarrow T_{\varphi(x,y)} M$ is injective.
\end{enumerate}
A map $\varphi$ as above is called a partially smooth parametrization of $\mathscr{L}$ at $p$. The elements $L$ of a partially smooth foliation $\mathscr{L}$ are called leaves.
\end{defn}

We shall use the symbol $L(p)$ to denote the leaf $L\in \mathscr{L}$ through $p$. According to the above definition, each leaf $L\in \mathscr{L}$ is a smooth embedded submanifold of $M$ and the map 
\[
p \mapsto T_p L(p)
\]
is a continuous map from $M$ into the Grassmannian of $k$-planes in $TM$. This map is not necessarily differentiable in directions that are not tangent to the leaves. 

\begin{rem}
One could modify this definition in order to allow the leaves to be only injectively immersed submanifolds, but this further level of generality is not needed in this paper.
\end{rem}

\begin{ex}
\label{cubic}
The partition of $\R^2$ into the graphs of the smooth functions $x\mapsto (x-x_0)^3$, for $x_0$ varying in $\R$, is a 1-dimensional partially smooth foliation of $\R^2$. Indeed, it has the global partially smooth parametrization
\[
\varphi(x,y) = \bigl( x,(x+y)^3 \bigr).
\]
Note that this is not a $C^1$ foliation. In fact, the $x$-axis is a smooth curve which is tangent to each leaf and contained in none. The presence of such a curve is not possible in a $C^1$ foliation, but it is possible in a partially smooth foliation, since the associated parametrizations are not assumed to be $C^1$ diffeomorphisms.
\end{ex}

\begin{ex}
\label{osculating}
An example with a special geometric flavour is described in \cite[Lecture 10, see in particular figure 10.10]{ft07}: Consider a smooth simple curve $\gamma$ in the plane with strictly increasing positive curvature. Then the osculating circles to $\gamma$, that is, the circles which are tangent to $\gamma$ with radius the inverse of the curvature of $\gamma$ at the point of tangency, form a foliation of a portion of the plane. This foliation is easily seen to be partially smooth. It is not a $C^1$ foliation, because the curve $\gamma$ is tangent to all the leaves of this foliation but touches each leaf in a single point.
\end{ex}

If $\mathscr{L}$ is a $k$-dimensional partially smooth foliation of $M$ and $\varphi_1: \Omega_1 \rightarrow U_1$ and $\varphi_2: \Omega_1 \rightarrow U_2$ are two partially smooth parametrizations of $\mathscr{L}$ with $U_1\cap U_2 \neq \emptyset$ then the transition map
\[
\psi:= \varphi_2^{-1} \circ \varphi_1 : \varphi_1^{-1}(U_2) \rightarrow \varphi_2^{-1}(U_1)
\]
is a homeomorphism between open subsets of $\R^k \times \R^{n-k}$ that maps subspaces parallel to $\R^k \times \{0\}$ into subspaces parallel to $\R^k \times \{0\}$ and hence has the form
\[
\psi(x,y) = \bigl( \psi_1(x,y), \psi_2(y) \bigr).
\]
Here, the first component $\psi_1$ has differentials of every order in the first variable $x\in \R^k$ that are continuous on $\varphi_1^{-1}(U_2)$, and
\[
d_1 \psi_1(x,y): \R^k \rightarrow \R^k
\]
is invertible for every $(x,y)\in \varphi_1^{-1}(U_2)$.

The next example gives an alternative way of presenting a partially smooth foliation.

\begin{ex}
\label{repgrafici}
Consider a family of graphs of maps from $\R^k$ into $\R^{n-k}$:
\[
\mathscr{L} = \{ \mathrm{graph} \, \sigma_y : \R^k \rightarrow \R^{n-k} \}_{y\in \R^{n-k}} 
\]
where 
\[
\sigma: \R^k \times \R^{n-k} \rightarrow \R^{n-k},
\]
and we are using the notation $\sigma_y(x)=\sigma(x,y)$. We assume that the elements of $\mathscr{L}$ form a partition of $\R^k \times \R^{n-k}$. If the map $\sigma$ is continuous on $\R^k \times \R^{n-k}$ and infinitely differentiable with respect to the first variable $x\in \R^k$ with all differentials $d^h_1\sigma$ continuous on $\R^k \times \R^{n-k}$, then $\mathscr{L}$ is a partially smooth foliation of $\R^k \times \R^{n-k}$. The reader is invited to check that every partially smooth foliation can be locally represented in this way by means of smooth charts.
\end{ex}

We now proceed by defining the natural notion of smoothness for maps whose domain carries a partially smooth foliation.

\begin{defn}
Let $\mathscr{L}$ be a partially smooth $k$-dimensional foliation of  $M$. A map $f: V \rightarrow N$ from an open subset $V$ of $M$ into a smooth manifold $N$ is said to be $\mathscr{L}$-smooth  if for any partially smooth parametrization $\varphi: \Omega \rightarrow U\subset V$ of $\mathscr{L}$ the composition $f\circ \varphi$ is continuous, infinitely differentiable in the first variable $x\in \R^k$ and all the partial differentials $d_1^h f$, $h\in \N$, are continuous on $\Omega$.  
\end{defn}

As usual, it is enough to check the $\mathscr{L}$-smoothness of a map $f: V \rightarrow N$ for a family of partially smooth parametrizations whose images cover $V$. Indeed, this follows from the properties of the transition maps that we have discussed before. 

If, more generally, the partially smooth foliation $\mathscr{L}$ is defined only on an open subset $W$ of $M$, we say that $f: V \rightarrow N$ as above is $\mathscr{L}$-smooth if its restriction to $V\cap W$ is $\mathscr{L}$-smooth.

The restriction of an $\mathscr{L}$-smooth map to each leaf of $\mathscr{L}$ is smooth, but derivatives in directions that are not tangent to the leaves may not exist.
 
Let $\mathscr{L}$ be a partially smooth foliation.  The inverse of each partially smooth parametrization for $\mathscr{L}$ is trivially seen to be $\mathscr{L}$-smooth. Moreover, the map
 \[
 p \mapsto T_p L(p)
\]
that maps each $p\in M$ to the tangent space of the leaf through $p$ is also $\mathscr{L}$-smooth. Indeed, its composition with a partially smooth parametrization $\varphi$ is just the map
\[
(x,y) \mapsto d_1 \varphi(x,y) (\R^k),
\]
which has the required regularity properties. 

Let $\mathscr{L}$ be a $k$-dimensional partially smooth foliation of the $n$-dimensional manifold $M$ and let $\mathscr{L}'$ be a $k'$-dimensional partially smooth foliation of the $n'$-dimensional manifold $M'$. Let $f: M \rightarrow M'$ be an $\mathscr{L}$-smooth map mapping each leaf of $\mathscr{L}$ into some leaf of $\mathscr{L}'$. Then for every $\mathscr{L}'$-smooth map $g: M' \rightarrow N$ into a manifold $N$ the composition $g\circ f: M \rightarrow N$ is $\mathscr{L}$-smooth. Indeed, $f$ can be read by means of partially smooth parametrizations of $\mathscr{L}$ and $\mathscr{L}'$ as a map
\[
(x,y) \mapsto \psi^{-1}\circ f \circ \varphi (x,y) = (f_1(x,y),f_2(y)), \qquad \forall (x,y)\in \Omega\subset \R^k \times \R^{n-k},
\]
where the continuous map $f_1$ takes values into $\R^{k'}$ and is infinitely differentiable in $x\in \R^k$ with continuous partial differentials of every order and $f_2$ is a continuous map taking values into $\R^{n'-k'}$. By means of the same partially smooth parametrization of $\mathscr{L}'$, the map $g$ is seen as a continuous map
\[
(x',y') \mapsto g\circ \psi(x',y'), \qquad \forall (x',y')\in \Omega' \subset \R^{k'} \times \R^{n'-k'},
\]
which is infinitely differentiable in $x'\in \R^{k'}$ with continuous partial differentials of every order. The chain rule implies that $g\circ f \circ \varphi$ has the required regularity properties.

\begin{defn}
Let $\mathscr{L}$ and $\mathscr{L}'$ be partially smooth foliations of open subsets $U$ and $U'$ of $M$. We say that $\mathscr{L}$ refines $\mathscr{L}'$ if for all $p\in U\cap U'$ the leaf $L(p)$ of $\mathscr{L}$ passing through $p$ is contained in the leaf $L'(p)$ of $\mathscr{L}'$ through $p$.
\end{defn}

In other words, $\mathscr{L}$ refines $\mathscr{L}'$ if any leaf of $\mathscr{L}$ is either disjoint from the domain of $\mathscr{L}'$ or fully contained in a leaf of $\mathscr{L}'$. 

Let $\mathscr{L}$ and $\mathscr{L}'$ be partially smooth foliations of open subsets $U$ and $U'$ of $M$ and assume that $\mathscr{L}$ refines $\mathscr{L}'$. In this case, any map $f: V \rightarrow N$ from an open subset $V$ of $M$ into a manifold $N$ that is $\mathscr{L}'$-smooth is a fortiori $\mathscr{L}$-smooth on $U'$: This follows from the above considerations about smoothness of compositions by factorizing the restriction of $f$ to $V\cap U\cap U'$ through the inclusion $V\cap U\cap U' \hookrightarrow V\cap U'$.

Let $\varphi$ and $\varphi'$ be partially smooth parametrizations  of $\mathscr{L}$ and $\mathscr{L}'$ respectively.  Then the inverse of $\varphi'$ is $\mathscr{L}$-smooth, because we have already observed that it is $\mathscr{L}'$-smooth. However, the map $\varphi^{-1}$ needs not be $\mathscr{L}'$-smooth: For instance, if $\mathscr{L}'=\{U'\}$ is the trivial $n$-dimensional foliation consisting of only one leaf, then $\mathscr{L}'$-smoothness is just ordinary smoothness, and the inverses of partially smooth parametrizations of $\mathscr{L}$ need not be smooth. These considerations motivate the following defintion:

\begin{defn}
Let $\mathscr{L}$ and $\mathscr{L}'$ be partially smooth foliations of open subsets $U$ and $U'$ of $M$ and assume that $\mathscr{L}$ refines $\mathscr{L}'$. The foliation $\mathscr{L}$ is said to be $\mathscr{L}'$-smooth if $U\cap U'$ can be covered by a family of open sets $V$ that are images of partially smooth parametrizations $\varphi: \Omega\rightarrow V$ for $\mathscr{L}$ such that 
 the map $\varphi^{-1}$ is $\mathscr{L}'$-smooth.
\end{defn}

Equivalently: the partially smooth foliation $\mathscr{L}$ refining $\mathscr{L}'$ is $\mathscr{L}'$-smooth if and only if the map 
\[
p \mapsto T_p L(p)
\]
is $\mathscr{L}$'-smooth. 
Indeed, if $\mathscr{L}$ is $\mathscr{L}'$-smooth and $\varphi: \Omega \rightarrow U$ is a partially smooth parametrization of $\mathscr{L}$ whose inverse is $\mathscr{L}'$-smooth, then the composition of the above map with a partially smooth parametrization of $\mathscr{L}'$ is the map
\[
(x,y) \mapsto d_1 \varphi (\varphi^{-1}\circ \psi(x,y)) (\R^h), \qquad (x,y)\in \R^k \times \R^{n-k}, \; h= \dim \mathscr{L},
\]
which has the required regularity properties. The converse statement can be proven by locally representing the foliation $\mathscr{L}$ as a family of graphs as in Example \ref{repgrafici}.

Now let $\mathscr{L}$ be a $k$-dimensional partially smooth foliation of $M$ and let $v:M \rightarrow TM$ be an $\mathscr{L}$-smooth vector field on $M$ that is tangent to each leaf of $\mathscr{L}$. 

This continuous vector field is not necessarily $C^1$ and its Cauchy problems may not have a unique solution. For instance, if $\mathscr{L}$ is the partially smooth foliation of $\R^2$ that is described in Example \ref{cubic} we can consider $v$ to be the unit vector field on $\R^2$ that is tangent to all the leaves and has a positive first component. The corresponding Cauchy problem at points of the form $(x,0)$ does not have a unique solution.

However, the restriction of $v$ to each leaf of $\mathscr{L}$ is a smooth tangent vector field, and hence the Cauchy problems for $v$ have unique solutions that are contained in some leaf. Therefore, the vector field $v$ has a well defined local flow 
\[
\phi_v : \mathrm{dom}(\phi_v) \rightarrow M, \qquad (t,x) \mapsto \phi_v^t(x) := \phi_v(t,x), \qquad \mathrm{dom}(\phi_v)  \subset \R\times M.
\]
The theorems on the continuous dependence of solutions of Cauchy problems depending on a parameter and on the differentiable dependence on the initial conditions imply that the maximal domain $\mathrm{dom}(\phi_v)$ is an open subset of $\R\times M$ and that the map $\phi_v$ is $\R\times \mathscr{L}$-smooth, where 
\[
\R\times \mathscr{L} := \{ \R \times L \}_{L\in \mathscr{L}}
\]
is a partially smooth foliation of $\R \times M$. 

\section{The graph transform}
\label{gtsec}

The graph transform in hyperbolic dynamics was originally introduced in order to prove the existence and the regularity properties of the local unstable manifold of a hyperbolic stationary point of a flow. It is also quite useful in order to understand the evolution of submanifolds that are transverse to the stable manifold of such a hyperbolic stationary point.
The aim of this section is to recall the definition and main properties of the graph transform and to establish its continuity in the $C^k$ topology.

Let $E$ be a finite dimensional real vector space. Let $\phi$ be a $C^1$ local flow on $E$ having zero as a hyperbolic stationary point. This means that the domain of $\phi$ is an open subset of $\R \times E$ containing $\R \times \{0\}$ and
\[
\phi^t(0) = 0, \qquad d\phi^t(0) = e^{tL} \qquad \forall t\in \R,
\]
where $L:E\rightarrow E$ is a linear mapping and $E$ has an $L$-invariant splitting
\[
E = E^u \oplus E^s
\]
such that
\[
\max \re \sigma(L|_{E^s}) <  0 < \min \re \sigma(L|_{E^u}).
\]
Denote by $P^u$ and $P^s$ the projectors onto $E^u$ and $E^s$ which are induced by this splitting.
As it is well known, $E$ admits an $L$-adapted norm, that is, a norm $\|\cdot\|$ such that
\[
\begin{split}
\|x\| &= \max\{\|P^u x\|,\|P^s x\|\}, \\ \|e^{tL} x\| \leq e^{-\lambda t} \|x\| \; \forall x\in E^s, \; \forall t\geq 0, & \qquad \|e^{tL} x\| \leq e^{\lambda t} \|x\| \; \forall x\in E^u, \; \forall t\leq 0,
\end{split}
\]
for some positive real number $\lambda$. Given a number $r>0$, we denote by $D(r)$ (resp.\ $D^u(r)$, resp.\ $D^s(r)$) the closed ball of radius $r$ in $E$ (resp.\ in $E^u$, resp.\ in $E^s$). With this notation we have
\[
D(r) = D^u(r) \times D^s(r).
\]
The local stable and unstable manifolds in $D(r)$ are the sets
\[
\begin{split}
W^s_{\mathrm{loc},r}(0) &:= \set{x\in D(r)}{\phi^t(x) \in D(r) \; \forall t\geq 0 \mbox{ and } \phi^t(x) \rightarrow 0 \mbox{ for } t\rightarrow +\infty}, \\
W^u_{\mathrm{loc},r}(0) &:= \set{x\in D(r)}{\phi^t(x) \in D(r) \; \forall t\leq 0 \mbox{ and } \phi^t(x) \rightarrow 0 \mbox{ for } t\rightarrow -\infty}.
\end{split}
\]
The symbol $\mathrm{Lip}_1( D^u(r), D^s(r))$ denotes the set of 1-Lipschitz maps from $D^u(r)$ to $D^s(r)$, which is a closed convex subset of the Banach space of continuous mappings from $D^u(r)$ to $E^s$. The map $\Gamma_{\phi}$ whose properties are listed in the next proposition is called the graph transform. 

\begin{prop}
\label{graph-transform}
For any $r>0$ small enough there is a continuous map
\[
\Gamma_{\phi} : [0,+\infty] \times \mathrm{Lip}_1( D^u(r), D^s(r)) \longrightarrow \mathrm{Lip}_1( D^u(r), D^s(r))
\]
with the following properties:
\begin{enumerate}[(i)]
\item $\Gamma_{\phi}$ is a semigroup, that is, $\Gamma_{\phi}(0,\sigma)=\sigma$ and $\Gamma_{\phi}(t+s,\sigma) = \Gamma(t,\Gamma(s,\sigma))$ for every $\sigma\in \mathrm{Lip}_1( D^u(r), D^s(r))$ and every $s,t\in [0,+\infty]$.
\item For every $t\in [0,+\infty)$ the restriction of $\phi^t$ to $D(r)$ maps the graph of $\sigma$ onto the graph of $\Gamma_{\phi}(t,\sigma)$, that is
\[
\graf \Gamma_{\phi}(t,\sigma) = \set{ \phi^t(x) }{x\in \graf \sigma \mbox{ and } \phi([0,t]\times \{x\}) \subset D(r)}.
\]
\item $\Gamma_{\phi}$ has a unique fixed point $\sigma_{\infty}$ such that $\Gamma_{\phi}(+\infty,\sigma) = \sigma_{\infty}$ for every 
\[
\sigma\in \mathrm{Lip}_1( D^u(r), D^s(r)),
\]
and 
\[
\graf \sigma_{\infty} = W^u_{\mathrm{loc},r}(0).
\]
\item If morever $\phi$ is of class $C^{k+1}$, with $k\geq 1$, then $\Gamma_{\phi}$ restricts to a mapping 
\[
[0,+\infty] \times \mathrm{Lip}_1\cap C^k( D^u(r), D^s(r)) \longrightarrow \mathrm{Lip}_1\cap C^k( D^u(r), D^s(r)),
\]
which is continuous with respect to the $C^k$ topology. In particular, $\sigma_{\infty}$ is of class $C^k$ and if $\sigma$ is of class $C^k$ then $\Gamma_{\phi}(t,\sigma)$ converges to $\sigma_{\infty}$ in the $C^k$-topology for $t\rightarrow +\infty$.
\end{enumerate}
\end{prop}

\begin{proof}
The proof of statements (i), (ii) and (iii) in the case of a discrete dynamical system can be found in \cite{shu87}[Section 5.I]. The case of a flow easily follows (see \cite{ama01}[Proposition A.3] for details about the dependence on time). Here we show how (iv) can be deduced from the first three statements applied to the tangential mapping of $\phi$. 

We start by recalling the explicit formula for graph transform $\Gamma_{\phi}(t,\sigma)$
\begin{equation}
\label{formula1}
\Gamma_{\phi}(t,\sigma)  = P^s \circ \phi^t \circ ( \mathrm{id}_{E^u} \times \sigma) \circ \bigl( P^u \circ \phi^t \circ ( \mathrm{id}_{E^u} \times \sigma) \bigr)^{-1}|_{D^u(r)} \qquad \forall t\in [0,+\infty).
\end{equation}
See \cite{shu87}[Definition 5.3]. Indeed, when $r$ is small enough the restriction to $D^u(r)$ of the map $P^u \circ \phi^t \circ ( \mathrm{id}_{E^u} \times \sigma)$ is a diffeomorphism onto a neighborhood of $D^u(r)$ (see \cite{shu87}[Lemma 5.5]). The identity (\ref{formula1}) implies that if $\phi$ is of class $C^k$ then $\Gamma_{\phi}$ restricts to a map
\[
[0,+\infty) \times \mathrm{Lip}_1\cap C^k( D^u(r), D^s(r)) \longrightarrow \mathrm{Lip}_1\cap C^k( D^u(r), D^s(r)),
\]
which is continuous with respect to the $C^k$ topology. There remains to show that 
if $\phi$ is of class $C^{k+1}$ then $\sigma_{\infty}$ is of class $C^k$ and for every $\sigma\in \mathrm{Lip}_1\cap C^k( D^u(r), D^s(r))$
\[
\Gamma_{\phi}(t,\sigma) \rightarrow \sigma_{\infty} \qquad \mbox{for } t\rightarrow +\infty
\]
in the $C^k$ topology, locally uniformly in $\sigma\in C^k( D^u(r), D^s(r))$. In order to show this, we start with the case $k=1$:

\medskip

\noindent{\em Claim 1.} If $\phi$ is of class $C^2$ then there is a positive number $r'\leq r$ such that $\sigma_{\infty}|_{D^u(r')}$ is of class $C^1$ and for every $\sigma\in \mathrm{Lip}_1\cap C^1( D^u(r'), D^s(r'))$
\[
\Gamma_{\phi}(t,\sigma) \rightarrow \sigma_{\infty}|_{D^u(r')} \qquad \mbox{for } t\rightarrow +\infty
\]
in the $C^1$ topology, locally uniformly in $\sigma\in C^1( D^u(r'), D^s(r'))$.

\medskip

The tangential mapping of $\phi^t$ is the map
\[
T\phi^t : TE:=E\times E \rightarrow TE, \qquad T\phi^t(x,u) := \bigl( \phi^t(x), d\phi^t(x)[u] \bigr),
\]
and defines a flow on $TE$. This flow is of class $C^k$ when $\phi$ is of class $C^{k+1}$. The point $(0,0)$ is an equilibrium point for this flow, and the differential of $T\phi^t$ at $(0,0)$ is
\[
d T\phi^t (0,0)[(u,v)] = \bigl(D\phi^t(0)[u], d^2 \phi^t(0)[0,u] + d\phi^t(0)[v]\bigr) = (e^{tL} u, e^{tL}v).
\] 
This shows that $(0,0)$ is hyperbolic, with corresponding projectors $P^u \times P^u$ and $P^s\times P^s$. Therefore, there exist a positive number $r'\leq r$ and a
continuous mapping
\[
\Gamma_{T\phi} : [0,+\infty] \times \mathrm{Lip}_1( D^u(r')\times D^u(r'), D^s(r')\times D^s(r')) \rightarrow \mathrm{Lip}_1( D^u(r')\times D^u(r'), D^s(r')\times D^s(r')),
\]
which satisfies the conditions (i), (ii) and (iii) for the flow $T\phi$. We claim that for every $\sigma\in \mathrm{Lip}_1\cap C^1( D^u(r'), D^s(r'))$ and every $t\in [0,+\infty)$ there holds\begin{equation}
\label{TGamma}
T \Gamma_{\phi}(t,\sigma) = \Gamma_{T\phi}(t,T\sigma).
\end{equation}
Indeed, the formula (\ref{formula1}) shows that the map $\Gamma_{\phi}(t,\sigma)$ is uniquely determined by the identity
\begin{equation}
\label{formula2}
\Gamma_{\phi}(t,\sigma) \circ P^u \circ \phi^t \circ ( \mathrm{id}_{E^u} \times \sigma) = P^s \circ \phi^t \circ ( \mathrm{id}_{E^u} \times \sigma).
\end{equation}
By applying the functor $T$ we find
\[
T\Gamma_{\phi}(t,\sigma) \circ TP^u \circ T\phi^t \circ T ( \mathrm{id}_{E^u} \times \sigma) = T P^s \circ T\phi^t \circ T ( \mathrm{id}_{E^u} \times \sigma),
\]
from which we obtain
\[
T\Gamma_{\phi}(t,\sigma) \circ (P^u\times P^u) \circ T\phi^t \circ ( \mathrm{id}_{E^u\times E^u} \times T\sigma) = (P^s\times P^s) \circ T\phi^t \circ ( \mathrm{id}_{E^u\times E^u} \times T \sigma).
\]
The last identity shows that $T\Gamma_{\phi}(t,\sigma)$ satisfies  (\ref{formula2}) for the flow $T\phi^t$ and hence proves (\ref{TGamma}). 

By the theorem on the limit under the sign of differential, the set
\[
\set{ T\sigma|_{D^u(r')\times D^u(r')} }{\sigma\in \mathrm{Lip}_1\cap C^1( D^u(r'), D^s(r'))}
\]
is a non-empty closed subset of $\mathrm{Lip}_1( D^u(r')\times D^u(r'), D^s(r')\times D^s(r'))$. Since this set is also invariant under the action of $\Gamma_{T\phi}$ by (\ref{TGamma}), statement (iii) implies that the unique fixed point of $\Gamma_{T\phi}$ belongs to it. Since the first component of $\Gamma_{T\phi}(t,T\sigma)$ converges uniformly to $\sigma_{\infty}|_{D^u(r')}$ for $t\rightarrow +\infty$, we deduce that $\sigma_{\infty}|_{D^u(r')}$ is continuously differentiable and $T\sigma_{\infty}|_{D^u(r')}$ is the fixed point of $\Gamma_{T\phi}$. The fact that $\Gamma_{T\phi}(t,T\sigma)$ $C^0$-converges to $ T\sigma_{\infty}|_{D^u(r')}$ for $t\rightarrow +\infty$ locally uniformly in $T\sigma\in C^0(D^u(r')\times D^u(r'), D^s(r')\times D^s(r'))$ implies Claim 1.

\medskip

By applying Claim 1 to $T\phi$, an induction argument proves the following:

\medskip

\noindent {\em Claim 2.} If $\phi$ is of class $C^{k+1}$ then there is a positive number $r'\leq r$ such that $\sigma_{\infty}|_{D^u(r')}$ is of class $C^k$ and for every $\sigma\in \mathrm{Lip}_1\cap C^k( D^u(r'), D^s(r'))$
\[
\Gamma_{\phi}(t,\sigma) \rightarrow \sigma_{\infty}|_{D^u(r')} \qquad \mbox{for } t\rightarrow +\infty
\]
in the $C^k$ topology, locally uniformly in $\sigma\in C^k( D^u(r'), D^s(r'))$.

\medskip

There remains to prove that the above claim remains true if we replace $r'$ by $r$. This follows from the identity (\ref{formula1}) together with a standard dynamical continuation argument. Indeed, since $\sigma_{\infty}$ is of class $C^k$ on $D^u(r')$ we can find $t_0>0$ large enough so that the map
\[
P^u \circ \phi^{t_0} \circ (\mathrm{id}_{E^u} \times \sigma_{\infty})|_{D^u(r')} : D^u(r') \rightarrow E^u
\]
is a diffeomorphism of class $C^k$ onto a neighborhood of $D^u(r)$ (see again \cite{shu87}[Lemma 5.5]). By (\ref{formula1}) together with the fact that $\sigma_{\infty}$ is a fixed point of $\Gamma_{\phi}$ we find
\[
\sigma_{\infty}  = P^s \circ \phi^{t_0} \circ ( \mathrm{id}_{E^u} \times \sigma_{\infty}) \circ \bigl( P^u \circ \phi^{t_0} \circ ( \mathrm{id}_{E^u} \times \sigma_{\infty}) \bigr)^{-1}|_{D^u(r)}.
\]
By the above choice of $t_0$, $\sigma_{\infty}$ on the right-hand side of this equality is applied at points in $D^u(r')$, so the regularity of $\sigma_{\infty}|_{D^u(r')}$ implies that $\sigma_{\infty}$ is of class $C^k$ on $D^u(r)$. Now let $\sigma_0 \in \mathrm{Lip}_1\cap C^k( D^u(r), D^s(r))$. By Claim 2 we can find a neighborhood $\mathscr{U}\subset \mathrm{Lip}_1\cap C^k( D^u(r), D^s(r))$ of $\sigma_0$ in the $C^k$ topology such that
\begin{equation}
\label{unif}
\Gamma_{\phi}(t,\sigma)|_{D^u(r')} \rightarrow \sigma_{\infty}|_{D^u(r')} \qquad \mbox{for } t\rightarrow +\infty \mbox{, uniformly in } \sigma\in \mathscr{U}.
\end{equation}
In particular, there exists $t_1\geq 0$ such that for every $t\geq t_1$ and every $\sigma\in \mathscr{U}$ the map
\[
P^u \circ \phi^{t_0} \circ (\mathrm{id}_{E^u} \times \Gamma_{\phi}(t-t_0,\sigma))|_{D^u(r')} : D^u(r') \rightarrow E^u
\]
is a diffeomorphism of class $C^k$ onto a neighborhood of $D^u(r)$. From the group property of $\Gamma_{\phi}$ we deduce the identity
\[
\Gamma_{\phi}(t,\sigma) = P^s \circ \phi^{t_0} \circ (\mathrm{id}_{E^u} \times \Gamma_{\phi}(t-t_0,\sigma) ) \circ \bigl( P^u \circ \phi^{t_0} \circ (\mathrm{id}_{E^u} \times \Gamma_{\phi}(t-t_0,\sigma)) \bigr)^{-1}|_{D^u(r)}.
\]
When $t\geq t_1$ and $\sigma\in \mathscr{U}$, the map $\Gamma_{\phi}(t-t_0,\sigma)$ on the right-hand side of this identity is applied at points in $D^u(r')$, so  (\ref{unif}) implies that $\Gamma_{\phi}(t,\sigma)$ converges to $\sigma_{\infty}$ in $C^k(D^u(r),D^s(r))$ uniformly for $\sigma\in \mathscr{U}$.

\end{proof}

\begin{rem}
Actually, $\sigma_{\infty}$ is of class $C^k$ whenever the flow is of class $C^k$. See \cite{shu87}[Section 5.II].
\end{rem}

\section{First properties of Morse-Smale gradient-like flows}
\label{gensec}

Let $M$ be a closed differentiable manifold and $\phi$ a smooth flow on $M$. The flow $\phi$ is said to be {\em gradient-like} if all its chain recurrent points are stationary points. Equivalently, $\phi$ admits a Lyapunov function, i.e. a continuous real valued function on $M$ which strictly decreases on all non-stationary orbits of $\phi$. The gradient-like flow $\phi$ is said to be {\em Morse} if all its stationary points are hyperbolic. Since hyperbolic stationary points are isolated, a Morse gradient-like flow has finitely many stationary points. The set of all stationary points of $\phi$ is denoted by $\mathrm{stat}(\phi)$. 

Given a hyperbolic stationary point $x$ of $\phi$ we denote by $E^u_x$ and $E^s_x$ the unstable and stable subspaces of $T_x M$ and by $\mathrm{ind}\, (x)$ the {\em Morse index} of $x$, i.e.\ the dimension of $E^u_x$. The  {\em unstable} and a {\em stable manifold} of $x$ are denoted by  
\[
\begin{split}
W^u(x) &:= \{ p\in M \mid \phi^t(p) \rightarrow x \mbox{ for } t\rightarrow -\infty\}, \\
W^s(x) &:= \{ p\in M \mid \phi^t(p) \rightarrow x \mbox{ for } t\rightarrow +\infty\}.
\end{split}
\]
They are smoothly embedded images of $E^u_x$ and $E^s_x$, respectively, and
\[
T_x W^u(x) = E^u_x, \qquad T_x W^s(x) = E^s_x.
\]
The sets $\{W^u(x)\}_{x\in \mathrm{stat}\, (\phi)}$ and $\{W^s(x)\}_{x\in \mathrm{stat}\, (\phi)}$ are two partitions on $M$. The Morse gradient-like flow $\phi$ is said to be {\em Smale} if for every pair $x,y$ of stationary points the unstable manifold of $x$ and the stable manifold of $y$ meet transversally.

Throughout the remaining part of this paper, $\phi$ is a Morse-Smale gradient-like smooth flow on a closed differentiable manifold $M$ of dimension $n$.

Let $x$ be a stationary point of $\phi$. We fix once and for all an identification of a neighborhood $N_x$ of $x$ with a neighborhood of $0$ in the vector space $E_x := T_x M \cong \R^n$, so that the stationary point $x$ is identified to the origin. We choose a norm on $E_x$ which is adapted to the differential of the flow at $x$, see Section \ref{gtsec}, and we denote by $D_x^u(r)$, $D^s_x(r)$ and $D_x(r) = D_x^u(r) \times D_x^s(r)$ the closed balls of radius $r$ in $E^u_x$, $E^s_x$ and $E_x$, respectively. 

We can choose this norm in such a way that its restrictions to both $E^u_x$ and $E^s_x$ are induced by scalar products on these spaces. Thanks to these facts, the sets $D^u_x(r)$ and $D^s_x(r)$ are ellipsoids and in particular have a smooth boundary.

We rescale the norm so that $D_x(1)$ is contained in the neighborhood of the origin in $E_x$ which is identified with $N_x$ and the graph transform of Proposition \ref{graph-transform} is well-defined for $r\leq 1$. In particular, we see $\{D_x(r)\}_{0<r\leq 1}$, as a fundamental system of neighborhoods of $x$ in $M$. 

It will be also convenient to choose the $C^1$ chart which identifies $N_x$ with a neighborhood of $0$ in $E_x$ in such a way that the local unstable and stable manifolds of $x$ are linear:
\[
W^s_{\mathrm{loc},r}(x) = D_x^s(r), \qquad W^u_{\mathrm{loc},r}(x)  = D_x^u(r),
\]
for every $r\in (0,1]$. Finally, we may assume that $\phi$ is transverse to the smooth hypersurfaces
\[
\partial D^u_x(1) \times D^s_x(1) \qquad \mbox{and} \qquad D^u_x(1) \times \partial D^s_x(1)
\]
and each orbits intersects them at most once.

The Morse-Smale condition has a number of standard consequences:
\begin{enumerate}[(MS-1)]
\item For every stationary point $x$ the set $\overline{W^u(x)} \setminus W^u(x)$  is the union of the unstable manifolds of some stationary points $y$ with $\mathrm{ind}(y) < \mathrm{ind}(x)$. In particular, the sets
\[
W_k^u := \bigcup_{\substack{x\in \mathrm{stat}(\phi) \\ \mathrm{ind} (x) = k}} W^u(x), \qquad k=0,1\dots,n,
\]
form a {\em smooth stratification} of $M$, meaning that $\{W^u_k\}_{k=1}^n$ is a partition of $M$, 
each $W^u_k$ is a smooth submanifold of $M$ of dimension $k$, and for every $h\in \{1,\dots, n\}$ we have
\[
\overline{W^u_h} \setminus W^u_h \subset \bigcup_{0\leq k < h} W^u_k,
\]
where the set on the right-hand side is closed. Analogous facts hold for the stable manifolds.

\item Up to reducing the size of the neighborhoods $D_x(1)$, we may assume that
\[
\phi( [0,+\infty) \times D_x(1)) \cap D_y(1) = \emptyset
\]
whenever $x$ and $y$ are distinct stationary  points of $\phi$ with $\ind(y) \geq \ind(x)$.

\end{enumerate}

We conclude this section by proving two consequences of the Morse-Smale condition, whose proofs uses the continuity of the graph transform in the $C^1$ topology. 

\begin{prop}
\label{serve}
Let $x,y$ be stationary points of $\phi$ and let $(p_h)$ be a sequence in $W^u(x)$ which converges to some $p$ in $W^u(y)$. Then for each $h\in \N$ there exist a linear subspace $V_h$ of $T_{p_h} W^u(x)$ of dimension $\mathrm{ind}(y)$ such that the sequence $(V_h)$ converges to $T_p W^u(y)$.
\end{prop}

\begin{proof}
By property (MS-1) above we have $\mathrm{ind}(y) \leq \mathrm{ind}(x)$ with equality if and only if $y=x$. We argue by induction on $\mathrm{ind}(x) - \mathrm{ind}(y)$. 

If $\mathrm{ind}(x) - \mathrm{ind}(y)=0$, then $y=x$ and $V_h := T_{p_h} W^u(x)$ converges to $T_p W^u(x)$ because $W^u(x)$ is an embedded submanifold. 

We assume the claim to be true when $\mathrm{ind}(x) - \mathrm{ind}(y)<k$ for some $k\geq 1$, and we consider the case $\mathrm{ind}(x) - \mathrm{ind}(y)=k$. It is enough to construct the linear subspaces with the desired asymptotic behaviour for a subsequence of $(p_h)$.
Up to applying $\phi^{-t}$ for $t$ sufficiently large, we may assume that the sequence $(p_h)$ belongs to the closed neighborhood $D_y(1)$ of $y$. Since $p_h$ is not in the unstable manifold of $y$, we can find $t_h\geq 0$ so that 
\[
q_h := \phi^{-t_h}(p_h) \in D^u_y(1) \times \partial D^s_y(1).
\]
Since $(p_h)$ converges to a point in the unstable manifold of $y$, the sequence $(t_h)$ tends to $+\infty$. Up to replacing $(p_h)$ with a subsequence, the sequence $(q_h)$ converges to some $q\in D_y((1)$. The limiting point $q$ must belong to $W^s(y)$, because any the forward orbit of any other point in $D_y(1)$ eventually leaves $D_y(1)$ and $(t_h)$ is unbounded. 

Let $z$ be the stationary point of $\phi$ such that $q$ belongs to $W^u(z)$. Being in $D^u_y(1) \times \partial D^s_y(1)$, $q$ does not coincide with $y$, so by the Morse-Smale assumption $\mathrm{ind}(z) > \mathrm{ind}(y)$. Therefore, $\mathrm{ind}(x) - \mathrm{ind}(z)< k$ and by the inductive assumption we can find a sequence of linear subspaces $W_h \subset T_{q_h} W^u(x)$ of dimension $\mathrm{ind}(z)$ which converges to $T_q W^u(z)$.

By the transversality of $W^u(z)$ and $W^s(y)$ in $q$ we can find a linear subspace $V$ of $T_q W^u(z)$ such that
\[
V \cap T_q W^s(y) = \{0\} \qquad \mbox{and} \qquad T_q M = V \oplus T_q W^s(y).
\]
Since $V$ is a direct summand of the tangent space to $W^s(y)$, its forward evolution with respect to the linearized flow tends to $T_y W^u(y)$, that is
\[
d \phi^t(q) V \rightarrow T_y W^u(y) \qquad \mbox{for } t\rightarrow +\infty,
\]
see e.g.\ \cite{ama03}[Theorem 2.1 (iii)]. In particular, we can find $T\geq 0$ such that $d\phi^T(q) V$ is the graph of a linear mapping from $E^u_y$ to $E^s_y$ of operator norm at most $1/3$. Up to the choice of a larger $T$, we can also assume that $\phi^T(q)$ belongs to $E_y(1/3)$. The sequence of linear subspaces 
\[
d\phi^T(q_h) W_h \subset T_{\phi^T(q_n)} W^u(x)
\]
converges to $d\phi^T(q) V$. Therefore, when $h$ is large enough the space $d\phi^T(q_h) W_h$ is the graph of a linear mapping $L_h : E^u_y \rightarrow E_y^s$ of operator norm at most $1/2$. If $h$ is large enough we also have $T\leq t_h$ and hence $\phi^T(q_h)$ belongs to $E_y(1)$. 

Let $(u_h,s_h)\in D_y^u(1) \times D_y^s(1)$ be the components of $\phi^T(q_h)$ in the splitting $E_y^u \times E^s_y$. We have the convergence
\[
(u_h,s_h) \rightarrow (\overline{u},\overline{s}) = \phi^T(q). 
\]
Since $\|\overline{s}\| \leq 1/3$ and $\|L_h\|\leq 1/2$, the map
\[
\sigma_h (u) := s_h + L_h (u-u_h)
\]
belongs to $\mathrm{Lip}_1(D^u_y(1),D^s_y(1))$ when $h$ is large enough. Up to a subsequence, $(\sigma_h)$ converges in the $C^1$-topology to a map
\[
\sigma(u) = \overline{s} + L ( u - \overline{u}),
\]
where the operator norm of $L$ is at most 1/2. By the continuity of the graph transform
\[
\Gamma_{\phi} : \mathrm{Lip}_1\cap C^k( D^u_y(1), D^s_y(1)) \longrightarrow \mathrm{Lip}_1\cap C^k( D^u_y(1), D^s_y(1)),
\]
which is proved in Proposition \ref{graph-transform}, we deduce that $\Gamma_{\phi}(t_h-T,\sigma_h)$ converges to $\Gamma(+\infty,\sigma) = \sigma_{\infty}$ in the $C^1$-topology, where
\[
\mathrm{graph} \, \sigma_{\infty} = W^u_{\mathrm{loc},1}(y) = W^u(y) \cap D_y(1).
\]
Together with the chain of identities
\[
\begin{split}
d\phi^{t_h}(q_h) W_h &= d\phi^{t_h-T}(\phi^T (q_h)) d\phi^T(q_h) W_h = d \phi^{t_h-T}(\phi^T(q_h)) \,\mathrm{graph} \, L_h \\ &= d \phi^{t_h-T}(\phi^T(q_h))\, \mathrm{graph} \, d\sigma_h(u_h) = d \phi^{t_h-T}(\phi^T(q_h)) T_{\phi^T(q_h)} \mathrm{graph}\, \sigma_h \\ &= T_{\phi^{t_h}(q_h)} \phi^{t_h-T} \bigl(\mathrm{graph}\, \sigma_h \bigr) = T_{p_h} \mathrm{graph}\, \Gamma_{\phi} (t_h-T,\sigma_h) \\ &= \mathrm{graph}\, d \Gamma_{\phi} (t_h-T,\sigma_h)( P^u p_h )
\end{split}
\]
where $P^u$ denotes the projection onto $E^u_y$ in the splitting $E_y = E^u_y \times E^s_y$, we obtain that $(d\phi^{t_h}(q_h) W_h)$ converges to
\[
\mathrm{graph}\, d \sigma_{\infty} (P^u p) = T_p \, \mathrm{graph} \, \sigma_{\infty} = T_p W^u(y),
\]
and $d\phi^{t_h}(q_h) W_h \subset T_{p_h} W^u(x)$ is the required sequence of linear subspaces.
\end{proof}

A corollary of the above result is that in a Morse-Smale gradient-like flow the intersections of stable and unstable manifolds of pairs of stationary points are {\em uniformly} transverse. In order to clarify this statement, we need to recall how transversality can be measured quantitatively. Let $V$ and $W$ be linear subspaces of an $n$-dimensional Euclidean space with $\dim V + \dim W \geq n$. Then we set
\[
\mathscr{T}(V,W) := \max_{\substack{V_0 \mathrm{\; subspace\; of\; } V \\ W_0 \mathrm{\; subspace\; of\;} W \\ \dim V_0 + \dim W_0 = n}} \min_{\substack{v\in V_0 \setminus \{0\}, \\ w\in W_0 \setminus \{0\} }} \mathrm{ang} (v,w) \in [0,\pi/2],
\]
where $\mathrm{ang} (v,w)\in [0,\pi]$ denotes the angle between the two non-vanishing vectors $v$ and $w$. The function $\mathscr{T}$ has the good properties which are required by a function measuring transversality of two linear subspaces:
\begin{enumerate}[($\mathscr{T}$-1)]
\item $\mathscr{T}(V,W) = \mathscr{T}(W,V)$;
\item $\mathscr{T}(V,W)>0$ if and only if $V$ and $W$ are transverse;
\item the function $\mathscr{T}$ is continuous in the Grassmannian topology;
\item if $V\subset V'$ then $\mathscr{T}(V,W)\leq \mathscr{T}(V',W)$.
\end{enumerate}
By using a Riemannian metric on $M$, the function $\mathscr{T}$ can be extended to pairs of linear subspaces of the tangent spaces of $M$.
We can now state the corollary about the uniformity of transversality for Morse-Smale gradient-like flows:

\begin{cor}
\label{uniftra}
The positive function
\[
M\rightarrow [0,\pi/2], \qquad p \mapsto \mathscr{T} \Bigl( T_p W^u(x(p)), T_p W^s(y(p)) \Bigr),
\]
where $x(p)$ and $y(p)$ are the unique stationary points such that $p\in W^u(x(p))\cap W^s(y(p))$, is lower semicontinuous. In particular, this function assumes a positive minimum on the compact manifold $M$.
\end{cor}

\begin{proof}
Let $(p_h)\subset M$ be a sequence converging to some $p\in M$. Up to a subsequence, we may assume that there are two stationary points $x$ and $y$ such that $x(p_h)=x$ and $y(p_h)=y$ for all $h\in \N$. Let $x'$ and $y'$ be the stationary points such that $p\in W^u(x')\cap W^s(y')$. By Proposition \ref{serve} there are sequences of subspaces $V_h^u\subset T_{p_h} W^u(x)$ and $V_h^s\subset T_{p_h} W^s(y)$ such that
\[
V_h^u \rightarrow T_p W^u(x') \qquad \mbox{and} \qquad V_h^s \rightarrow T_p W^s(y').
\]
Then
\[
\mathscr{T} \Bigl( T_{p_h} W^u(x), T_{p_h} W^s(y) \Bigr) \geq \mathscr{T}(V_h^u,V_h^s),
\]
and, since the latter quantity converges to $\mathscr{T}( T_p W^u(x'), T_p W^s(y'))$, we obtain 
\[
\liminf_{h\rightarrow \infty} \mathscr{T} \Bigl( T_{p_h} W^u(x), T_{p_h} W^s(y) \Bigr) \geq \mathscr{T}\Bigl( T_p W^u(x'), T_p W^s(y')\Bigr).
\]
This proves the lower semicontinuity.
\end{proof}

\section{Stable and unstable foliations}
\label{folsec}

Let $x$ be a stationary point of $\phi$ of Morse index $k$. Let $\Sigma$ be a smooth hypersurface in $M$ that is transverse to the flow $\phi$ and meets the stable manifold $W^s(x)$ transversally at some stable sphere $\partial D^s_x(r)$, $0<r\leq 1$. Up to reducing the size of $\Sigma$, it is easy to endow it with a smooth foliation 
\[
\{L^u(q)\}_{q\in \partial D^s_x(r)}
\] 
consisting of $k$-dimensional open embedded disks such that each $L^u(q)$ meets $W^s(x)$ transversally at $q$. If we evolve $\Sigma$ and its foliation by the flow $\phi$ we obtain - after possibly reducing the size of $\Sigma$ and its foliating disks - a smooth invariant foliation 
\[
\{L^u(q)\}_{q\in W^s(x)\setminus \{x\}}
\] 
of an invariant neighborhood of $W^s(x)\setminus \{x\}$. Here, the leaf $L^u(q)$ meets $W^s(x)$ transversally at $q$. By using the graph transform one can show that the leaves $L^u(q)$ tend to the unstable manifold $W^u(x)$ for $q\rightarrow x$, and hence one can set $L(x):= W^u(x)$ and obtain an invariant foliation of an invariant neighborhood $U$ of $W^u(x)\cup W^s(x)$. This foliation 
\[
\{L^u(q)\}_{q\in W^s(x)}
\] 
is known as an unstable foliation for $x$, and $U$ can be called an invariant tubular neighborhood of $W^s(x)$, as it carries a retraction  
\[
\pi^s : U_x \rightarrow W^s(x)
\]
mapping the point of each leaf $L^u(q)$ into its base points $q\in W^s(x)$. In general, this unstable foliation is not smooth at the points of $W^u(x)$, but it
turns out to be a partially smooth foliation in the sense of Definition \ref{admissible}. By inverting the time arrow, one obtains stable foliations 
\[
\{L^s(p)\}_{p\in W^u(x)}
\]
on a tubular neighborhood $U$ of $W^u(x)$. Such a stable foliation can be chosen to be smooth on $U\setminus W^s(x)$,  but will be just a partially smooth foliation on the whole $U$.

In the next section, we will construct stable foliations on tubular neighborhoods $U_x$ of $W^u(x)$ for each stationary point $x$ of $\phi$ in such a way that each stable foliation of dimension $k$ refines all the other stable foliations of dimension $k'>k$. These foliations will be constructed inductively on the Morse index, starting from the stationary points of index $0$. The lack of smoothness of the stable foliation of $U_x$ at points sitting on $W^s(x)$ will propagate in this inductive construction, so that when $\ind (y)>\ind (x)$ the stable foliation on $U_y$ will not be smooth outside of $W^s(y)$, but will be nevertheless a partially smooth foliation. This explains the regularity assumptions in the next definition and in the subsequent proposition.

\begin{defn}
Let $x$ be a stationary point of $\phi$ with Morse index $k$. An  {\em invariant tubular neighborhood} of $W^s(x)$ consists of an  invariant open neighborhood $U$ of $W^s(x)$ together with a continuous retraction $\pi^s : U \rightarrow W^u(x)$ such that:
\begin{enumerate}[(i)]
\item The fibers of $\pi^s$, which are denoted by
\[ 
L^u(q):= (\pi^s)^{-1}(\{q\}), \qquad q\in W^s(x),
\]
form a partially smooth $k$-dimensional foliation $\mathscr{L}^u$ of $U$. Each leaf $L^u(q)$ is transverse to $W^s(x)$ at the point $q$.
This foliation is called {\em unstable foliation} and its leaves {\em unstable leaves}.
\item The unstable foliation is invariant under the flow $\phi$. More precisely:
\[
\phi^t(L^u(q)) = L^u(\phi^t(q)) \qquad \forall t\in \R, \; \forall q\in W^s(x).
\]
In particular, $L^u(x)=W^u(x)$.

\end{enumerate}
By reversing the role of time, we obtain the definition of  invariant tubular neighborhood 
\[
\pi^u: U \rightarrow W^u(x)
\]
of $W^u(x)$ and of stable foliation
\[
\mathscr{L}^s=\{L^s(p)\}_{p\in W^u(x)} = \{ (\pi^u)^{-1}(\{p\}) \}_{p\in W^u(x)}.
\]
\end{defn}

Notice that, by invariance, the neighborhood $U$ of $W^s(x)$ (or of $W^u(x)$) in the above definition is actually a neighborhood of $W^u(x)\cup W^s(x)$. By reducing the invariant neighborhood $U$, one can of course obtain that the leaves $L^u(q)$ (or $L^s(p)$) are diffeomorphic to open disks, but we do not require this in the above definition.

The existence of invariant tubular neighborhoods and of the corresponding stable foliations with some further compatibility properties will be established in the next section by using the following results. 

\begin{prop} 
\label{stafol-prop}
Let $\tilde{U}$ be an invariant open neighborhood of $W^s(x)\setminus \{x\}$ that is contained in $M\setminus W^u(x)$ and let 
\[
\{ L^u(q) \}_{q\in W^s(x) \setminus \{x\}} 
\]
be a $\phi$-invariant $k$-dimensional partially smooth foliation of $\tilde{U}$ such that each leaf $L^u(p)$ is transverse to $W^s(x)$ at $q$. Set
\[
U:= \tilde{U}\cup W^u(x), \qquad L^u(x):= W^u(x).
\]
Then $U$ is an invariant tubular neighborhood of $W^s(x)$ with the retraction
\[
\pi^s : U \rightarrow W^s(x), \qquad p\mapsto q \;\; \mbox{if } p\in L^u(q), 
\]
and the unstable foliation
\[
\mathscr{L}^u := \{ L^u(q) \}_{q\in W^s(x)}.
\]
\end{prop}

\begin{proof}
We have to prove that the invariant set $U$ is open and that $\mathscr{L}^u$ is a partially smooth foliation of it. Indeed, once this is proven the continuity of the retraction $\pi^s$ follows from the fact that each leaf $L^u(q)$ is transverse to $W^s(x)$ at $q$.

By assumption, $\mathscr{L}^u$ restricts to a partially smooth foliation of $\tilde{U}$, so we just need to consider a point $p\in W^u(x)$ and show that a $U$ contains an open neighborhood of $p$ and $\mathscr{L}^u$ is a partially smooth foliation on this neighborhood. By the flow invariance, we may move $p$ on its $\phi$-orbit and assume that it is arbitrarily close to $x$.

Recall that we are identifying a neighborhood of $x$ with the product $D_x(1) = D_x^u(1) \times D_x^s(1)$, where $D^u_x(1)$ resp.\ $D^s_x(1)$ is the unit ball in the unstable space $E^u_x$ resp.\ stable space $E^s_x$, in such a way that
\[
W^u(x) \cap D_x(1) = D^u_x(1) \times \{0\}, \qquad W^s(x) \cap D_x(1) = \{0\} \times D^s_x(1).
\]

We claim that there is a number $0<r<1/2$ such that for every $q\in D^s_x(r)$ the intersection of the leaf $L^u(q)$ with $D_x(2r)$ is the graph of a smooth map
\[
\sigma_q: D^u_x(2r) \rightarrow D^s_x(2r)
\]
that is $1/2$-Lipschitz and satisfies $\sigma_q(0)=q$. 

This is obviously true for $q=0$, as we can take $\sigma_0$ to be the zero map. Any other $q\in D^s_x(1)$ is the evolution for some non-negative time of some $q_0\in \partial D^s_x(1)$. By assumption, the tangent space of $L^u(q_0)$ at $q_0$ is a $k$-dimensional subspace of $E_x^u \times E^s_x$ that is transverse to $(0) \times E^s_x$. This fact implies that the evolution of this space under the differential of the flow converges to the unstable space $E^u_x$ for $t\rightarrow +\infty$:
\[
d\phi^t(q_0) T_{q_0} L^u(q_0) \rightarrow E^u_x \times (0) \qquad \mbox{for } t\rightarrow +\infty,
\]
see for instance \cite[Theorem 2.1]{ama03}. This convergence is uniform in $q_0\in \partial D^s_x(1)$. Therefore, the invariance property of the foliation $\{L^u(q)\}_{q\in W^s(x) \setminus \{x\}}$ implies that there exists a positive number $r_1<1$ such that for every $q\in D^s_x(r_1)\setminus \{0\}$ the tangent space of $L^u(q)$ at $q$ is the graph of a linear map from $E^u_x$ to $E^s_x$ of operator norm smaller than $1/3$. Then the germ at $q$ of each leaf $L^u(q)$, $q\in D^s_x(r_1)\setminus \{0\}$, is the graph of a smooth map from a small disk around $0$ in $E^u_x$ into a small disk around $0$ in $E^s_x$ that is $1/2$-Lipschitz. By using standard results on how the flow near the hyperbolic stationary point $x$ acts on Lipschitz graphs, we obtain the claim (see e.g. \cite{shu87}[Lemma 5.6] or \cite{ama01}[Addendum A.5]). 

By the claim, 
\[
L^u(q) \cap D_x(2r) = \{ (p,\sigma_q(p)) \mid p\in D^u_x(2r)\}, \qquad \forall q\in D^s_x(r).
\]
Denote by 
\[
\tau: D^s_x(r) \rightarrow [0,+\infty]
\]
the continuous function that is defined by
\[
\phi^{-\tau(q)} \in \partial D^s_x(r) \quad \forall q\in D^s_x(r) \setminus \{0\}, \qquad \tau(0):= +\infty.
\]
By the flow invariance, we have
\[
\sigma_q = \Gamma(\tau(q),\sigma_{\phi^{-\tau(q)} (q)}) \qquad \forall q\in D^s_x(r) \setminus \{0\},
\]
where $\Gamma$ denotes the graph transform on 1-Lipschitz maps from $D^u_x(2r)$ to $D^s_x(2r)$. Moreover,
\[
\sigma_0 = 0 = \Gamma(+\infty,\sigma_q) \qquad \qquad  \forall q\in D^s_x(r).
\]
By the admissibility of the foliation $\mathscr{L}$ near $\partial D^s_x(r)$ and by regularity property of the graph transform that is stated in Proposition \ref{graph-transform} (iv), the map
\[
D^u_x(2r) \times D^s_x(r) \rightarrow D^s_x(2r), \qquad (p,q) \mapsto \sigma_q(p),
\]
is continuous, infinitely differentiable in the first variable $p\in D^u_x(2r)$ and all its partial differentials $d^h\sigma_q(p)$ in this variable are continuous on $D^u_x(2r) \times D^s_x(r)$. Therefore, the homeomorphism
\[
(p,q) \mapsto (p,\sigma_q(p))
\]
is a partially smooth parametrization of the foliation $\mathscr{L}^u$ at every point in the interior of $D^u_x(2r)$, as in Remark \ref{repgrafici}.
\end{proof}

Let $x$ be a stationary point of index $k$ and let
\[
\pi^u: U \rightarrow W^u(x) \qquad \mbox{and} \qquad \pi^s: U' \rightarrow W^s(x)
\]
be invariant tubular neighborhoods of $W^u(x)$ and $W^s(x)$, respectively, with stable and unstable foliations
\[
\mathscr{L}^s = \{L^s(p)\}_{p\in W^u(x)} \qquad \mbox{and} \qquad \mathscr{L}^u = \{L^u(p)\}_{q\in W^s(x)}.
\]
We assume the unstable foliation $\mathscr{L}^u$ to be smooth on $U'\setminus W^u(x)$, whereas $\mathscr{L}^s$ is just assumed to be partially smooth. Our aim now is to show that these two invariant foliations define a {\em product structure} on an invariant open subset of $U\cap U'$, that allow us to identify this set with an open subset of $W^u(x) \times W^s(x)$.

Arguing as in the proof of Proposition \ref{stafol-prop}, we find a small positive number $r< 1/2$ such that:
\begin{enumerate}[($\sigma$-1)]
\item For every $q\in D^s_x(r)$ the intersection of $L^u(q)$ with $D_x(2r)$ is the graph of a smooth map
\[
\sigma_q^u : D_x^u(2r) \rightarrow D_x^s(2r)
\]
that is $1/2$-Lipschitz and satisfies $\sigma_q^u(0)=q$.
\item For every $p\in D^s_x(r)$ the intersection of $L^s(p)$ with $D_x(2r)$ is the graph of a smooth map
\[
\sigma_p^s : D_x^s(2r) \rightarrow D_x^u(2r)
\]
that is $1/2$-Lipschitz and satisfies $\sigma_p^s(0)=p$.
\end{enumerate}

Fix $p\in D^u_x(r)$ and $q\in D^s_x(r)$. Being graphs of maps with Lipschitz constant smaller than 1, the graphs of $\sigma_p^u$ and $\sigma_q^s$ have a unique intersection point, that we denote by
\[
[p,q] := \mathrm{graph} (\sigma_p^u) \cap \mathrm{graph} (\sigma_q^s) \in D_x(2r).
\]
In other words, $[p,q]$ denotes the unique intersection point of the leaves $L^u(p)$ and $L^s(q)$ in $D_x(2r)$. The map
\begin{equation}
\label{prodloc}
D^u_x(r) \times D^s_x(r) \rightarrow M, \qquad (p,q) \mapsto [p,q],
\end{equation}
is easily seen to be a homeomorphism onto its image. We denote by $V_0$ the open set
\[
V_0 := \{ [p,q] \mid (p,q)\in \mathrm{int}(D^u_x(r))\times \mathrm{int}(D^s_x(r))\},
\]
and by $V$ the invariant open set that is generated by $V_0$:
\[
V := \phi( \R \times V_0).
\]
The map (\ref{prodloc}) extends by dynamical continuation to the invariant set $V$: The set
\[
\Omega := \{ (\phi^t(p), \phi^t(q)) \mid (p,q)\in \mathrm{int}(D^u_x(r))\times \mathrm{int}(D^s_x(r)), \; t\in \R\}
\]
is an open neighborhood of $(W^u(x)\times \{x\}) \cup (\{x\} \times W^s(x))$ in $W^u(x)\times W^s(x)$ and the identity
\[
[\phi^t(p),\phi^t(q)] = \phi^t([p,q]) \qquad \forall  (p,q)\in \mathrm{int}(D^u_x(r))\times \mathrm{int}(D^s_x(r)), \; t\in \R,
\]
defines a homeomorphism
\begin{equation}
\label{homeo}
\Omega \rightarrow V, \qquad (p,q) \mapsto [p,q],
\end{equation}
which is a conjugacy between the diagonal flow $\phi^t\times \phi^t$ on $\Omega$ and the flow $\phi$ on $V$. This map is said to be a {\em product structure} on $V$

It is useful to have a better understanding of the open set $\Omega$. To this purpose, we define the function
\[
\tau^u: W^u(x) \rightarrow [-\infty,+\infty)
\]
by
\begin{equation}
\label{tauu}
\tau^u(p):= \inf \{ t\in \R \mid p \in \phi^t(D^u_x(r)) \}.
\end{equation}
The function $\tau^u$ is continuous, smooth on $W^u(x)\setminus \{x\}$ and satisfies the following conditions:
\[
\tau^u(p) = -\infty \; \Leftrightarrow \; p=x, \quad D^u_x(r) = \{p\in W^u(x) \mid \tau^u(p) \leq 0\}, \quad \tau^u(\phi^t(p)) = \tau^u(p) + t,
\]
for every $p\in W^u(x)$ and $t\in \R$. Similarly, we introduce the function
\begin{equation}
\label{taus}
\tau^s: W^s(x) \rightarrow (-\infty,+\infty], \qquad
\tau^s(q):= \sup \{ t\in \R \mid q \in \phi^t(D^s_x(r)) \},
\end{equation}
which is continuous, smooth on $W^s(x)\setminus \{x\}$, and satisfies
\[
\tau^s(q) = +\infty \; \Leftrightarrow \; q=x, \quad D^s_x(r) = \{q\in W^s(x) \mid \tau^s(q) \geq 0\}, \quad \tau^s(\phi^t(q)) = \tau^s(q) + t,
\]
for every $q\in W^s(x)$ and $t\in \R$. From the identity
\[
\mathrm{int}(D^u_x(r))\times \mathrm{int}(D^s_x(r)) = \{ (p,q) \in W^u(x) \times W^s(x) \mid \tau^u(p) < 0 < \tau^s(q) \}
\]
we deduce that $\Omega$ has the form
\begin{equation}
\label{formomega}
\Omega = \{ (p,q) \in W^u(x) \times W^s(x) \mid \tau^u(p) < \tau^s(q) \}.
\end{equation}

The homeomorphism (\ref{homeo}) is in general not differentiable. However, the fact that $\mathscr{L}^u$ is a smooth foliation outside $W^u(x)$ and that each leaf of $
\mathscr{L}^s$ is smooth implies that for every $p\in W^u(x)$ the restriction
\[
q \mapsto [p,q]
\]
is a smooth embedding from the punctured disk 
\[
\{ q\in W^s(x) \setminus \{x\} \mid (p,q)\in \Omega\} = \{ q\in W^s(x) \mid \tau^u(p) < \tau^s(q) < +\infty\}
\]
into $M$. Moreover, the fact that $\mathscr{L}^s$ is a partially smooth foliation implies that the differentials of every order in the variable $q\in W^s(x)$ of the map (\ref{homeo}) are continuous on $\Omega \setminus (W^u(x) \times \{x\})$.

Building on the above facts, it is now easy to show that the stable leaves $L^s(p)$ can be further foliated into spheres of codimension one that are centered at the based point $p\in W^u(x)$. More precisely, we have the following result.

\begin{prop}
\label{folinsphe}
Let $x$ be a stationary point of index $k$, let $\pi^u: U \rightarrow W^u(x)$ be an invariant tubular neighborhood of $W^u(x)$ and let 
\[
\mathscr{L}^s = \{L^s(p)\}_{p\in W^u(x)}
\]
be the corresponding stable foliation, whose leaves have dimension $n-k$. Then there is an invariant neighborhood $V\subset U$ of $x$ and an invariant partially smooth foliation
\[
\mathscr{S}^s = \{ S^s(p,a)\}_{(p,a)\in \Lambda}
\]
of $V\setminus W^u(x)$ into $(n-k-1)$-dimensional spheres that refines the foliation $\mathscr{L}^s$ and is $\mathscr{L}^s$-smooth. Here, $\Lambda$ is an open neighborhood of $\{x\} \times \R$ in $W^u(x)\times \R$ and the invariance property reads
\begin{equation}
\label{invafoli}
\phi^t(S^s(p,a)) = S^s(\phi^t(p),a+t), \qquad \forall (p,a) \in \Lambda, \; t\in \R.
\end{equation}
Moreover, there is a one-parameter family of invariant open neighborhoods $V(r)$, $r\in (0,1]$, of $x$ such that:
\begin{enumerate}[(i)]
\item $V(1)= V$;
\item $\overline{V(s)} \subset V(r) \cup \overline{W^u(x)\cup W^s(x)}$ if $0<s<r\leq 1$;
\item $\displaystyle \bigcap_{r\in (0,1]} V(r) = W^u(x) \cup W^s(x)$;
\item each leaf of $\mathscr{S}^s$ is either contained in $V(r)$ or disjoint from it, for every $r\in (0,1]$.
\end{enumerate} 
\end{prop}

\begin{proof}
Let $\pi^s: U' \rightarrow W^u(x)$ be a tubular neighborhood of $W^s(x)$ such that the associated unstable foliation
\[
\mathscr{L}^u = \{L^u(q)\}_{q\in W^s(x)}
\]
is smooth on $U'\setminus W^u(x)$. Let 
\[
[\cdot,\cdot]: \Omega \rightarrow V
\]
be the product structure that the foliations $\mathscr{L}^u$ and $\mathscr{L}^s$ induce on an open invariant subset $V$ of $U\cap U'$ containing $x$, as in the considerations above. Denote by $\tau^u$ and $\tau^s$ the functions that are defined in (\ref{tauu}) and (\ref{taus}).

We consider a smooth invariant foliation of $W^s(x)\setminus \{x\}$ into $(n-k-1)$-dimensional spheres that are transverse to the flow, such as for instance
\[
S^s(x,a) := \phi^a ( \partial D^s_x(r) ) = \{ q\in W^s(x) \mid \tau^s(q) = a\}, \quad a\in \R.
\]
By using the product structure, we extend it to a foliation of $V\setminus W^u(x)$:
\begin{equation}
\label{leafS}
\mathscr{S}^s := \{S^s(p,a)\}_{(p,a)\in \Lambda}, \qquad 
S^s(p,a) := \{ [p,q] \mid q\in S^s(x,a)\} = \{ [p,q] \mid \tau^s(q)=a\},
\end{equation}
where
\[
\Lambda := \{ (p,a) \in W^u(x) \times \R \mid \tau^u(p) < a \} 
\]
is an open neighborhood of $\{x\}\times \R$ in $W^u(x)\times \R$. The regularity properties of the product structure imply that $\mathscr{S}^s$ is a partially smooth foliation and is $\mathscr{L}^s$-smooth. The conjugacy property of the product structure implies that the foliation $\mathscr{S}^s$ has the invariant property (\ref{invafoli}).

By using the representation (\ref{formomega}) for $\Omega$, we can define the fundamental system of neighborhoods $V(r)$, $r\in (0,1]$, of $W^u(x) \cup W^s(x)$ simply by
\[
V(r) := \{ [p,q] \mid \tau^u(p) < \tau^s(q) + \log r \}.
\]
Properties (i)-(iv) immediately follow.
\end{proof}

We conclude this section by constructing a suitable function on the invariant set $V$ of the above proposition which will be useful in Section \ref{finalsec}.

\begin{prop}
\label{coerfu}
Keeping the notations of Proposition \ref{folinsphe}, there exists a continuous function
\[
\rho: V \rightarrow [0,+\infty)
\]
with the following properties:
\begin{enumerate}[(i)]
\item $\rho^{-1}(\{0\}) = W^u(x)$.
\item $\rho$ diverges at every point in $\partial V\setminus \overline{W^u(x)}$.
\item $\rho$ is constant on each leaf of $\mathscr{S}^s$.
\item The restriction of $\rho$ to $V\setminus W^u(x)$ is $\mathscr{L}^s$-smooth. Moreover, if the foliation $\mathscr{L}^s$ is $\mathscr{L}'$-smooth with respect to some other foliation $\mathscr{L}'$ that is refined by $\mathscr{L}^s$, then $\rho|_{V\setminus W^u(x)}$ is $\mathscr{L}'$-smooth.
\item The restriction of $\rho$ to each punctured leaf $L^s(p)\setminus \{p\}$, $p\in W^u(x)$, has no critical points.
\item $\rho$ is differentiable along the flow $\phi$ and the identity
\[
\frac{d}{dt} \rho\circ \phi^t \Bigr|_{t=0} = - \rho
\]
holds on $V$.
\end{enumerate}
\end{prop}

\begin{proof}
The proof makes again use of the product structure (\ref{homeo}) on $V$ and the functions $\tau^u$ and $\tau^s$ that have been introduced before Proposition \ref{folinsphe}. Let $\sigma: (0,+\infty) \rightarrow [0,+\infty)$ be a smooth function such that
\[
\lim_{s\rightarrow 0} \sigma(s) = +\infty, \qquad \sigma(s) = 0 \;\; \forall s\geq 1, \qquad \sigma'(s) \leq 0\;\; \forall s>0.
\]
We extend $\sigma$ to a continuous function on $[0,+\infty]$ by setting
\[
\sigma(0):= +\infty \qquad \mbox{and} \qquad \sigma(+\infty) = 0.
\]
By (\ref{formomega}), the difference $\tau^s(q) - \tau^u(p)$ is positive on $\Omega$, so the function
\[
\tilde{\rho} : \Omega \rightarrow [0,+\infty), \qquad \tilde{\rho} (p,q) = \exp \bigl( - \tau^s(q) + \sigma ( \tau^s(q) - \tau^u(p) ) \bigr).
\]
is well defined and continuous. We define the continuous function $\rho: V \rightarrow  [0,+\infty)$ by using the product structure:
\[
\rho([p,q]) = \tilde{\rho}(p,q).
\]
The argument of the exponential function defining $\tilde{\rho}$ is $-\infty$ if and only if $\tau^s(q)=+\infty$, that is, if and only if $q=x$. This shows that the non-negative function $\tilde{\rho}$ vanishes precisely at $W^u(x) \times \{x\}$, and hence $\rho$ satisfies (i). If a sequence $(p_h,q_h)\subset \Omega$ converges to a point in $\partial \Omega$ then  $(\tau^s(q_h)-\tau^u(p_h))$ is infinitesimal and $(\tau^s(q_h))$ converges (see again (\ref{formomega})). By the properties of $\sigma$, $\tilde{\rho}(p_n,q_n)$ tends to $+\infty$ and hence $\rho$ satisfies (ii). The fact that $\tilde{\rho}$ is a function of $\tau^u$ and $\tau^s$ and the form (\ref{leafS}) of the leaves of $\mathscr{S}^s$ implies (iii). The function $\tilde{\rho}$ is smooth on $\Omega \setminus (W^u(x) \times \{x\})$, so the regularity of the product structure implies that $\rho$ is $\mathscr{L}^s$-smooth on $V \setminus W^u(x)$, as claimed in (iv). For the same reasons, this function inherits any further smoothness property that the foliation $\mathscr{L}^s$ might have. For every $p\in W^u(x)$, the punctured leaf $L^s(p)\setminus \{p\}$ is the image by the product structure of the $(n-k)$-dimensional punctured disk
\begin{equation}
\label{pd}
\{p\} \times \{q\in W^s(x) \mid \tau^u(p) < \tau^s(q) < +\infty\},
\end{equation}
which is foliated by the one-parameter family of spheres
\[
\{p\} \times \{q\in W^s(x) \mid \tau^s(q) = a\}, \qquad a\in (\tau^u(p) , +\infty).
\]
The value of $\tilde{\rho}$ on each of the above spheres is
\[
\exp( - a + \sigma(a-\tau^u(p)) )
\]
and differentiation in $a$ yields
\[
(-1+\sigma'(a-\tau^u(p))) \exp( - a + \sigma(a-\tau^u(p)) ).
\]
Since $\sigma'\leq 0$, the above number is strictly negative. This implies that the restriction of $\tilde{\rho}$ to the punctured disk (\ref{pd}) has no critical points and hence (v) holds. By using the properties (\ref{tauu}) and (\ref{taus}) of the functions $\tau^u$ and $\tau^s$ we obtain
\[
\tilde{\rho}(\phi^t(p),\phi^t(q)) = \exp \bigl( - \tau^s(q) - t + \sigma ( \tau^s(q) - \tau^u(p) ) \bigr),
\]
so differentiation in $t$ yields
\[
\frac{d}{dt} \tilde{\rho}(\phi^t(p),\phi^t(q)) \Bigr|_{t=0} = - \exp \bigl( - \tau^s(q)  + \sigma ( \tau^s(q) - \tau^u(p) ) \bigr) = - \tilde{\rho}(p,q).
\]
Then (vi) follows from the fact that the product structure conjugates the diagonal flow $\phi^t\times \phi^t$ on $\Omega$ with the flow $\phi$ on $V$.
\end{proof}

\begin{rem}
\label{fund}
Statement (i) of this proposition, together with the continuity of $\pi^u$ and $\rho$, implies that for every $p\in W^u(x)$ and every neigborhood $N$ of $p$ in $M$ we can find an open neighborhood $P$ of $p$ in $W^u(x)$ and a positive number $\epsilon$ such that the open set
\[
\rho^{-1}([0,\epsilon)) \cap \bigcup_{p'\in P} L^s(p') 
\]
is contained in $N$.
\end{rem}

\section{Stable foliations with the refinement property} 
\label{refsec}

The main result of this section is the following existence result for stable foliations $\mathscr{L}^s_x$ for each stationary point $x$ with some further compatibility properties, saying that $\mathscr{L}^s_x$ is refined by a foliation $\mathscr{S}^s_x$ into spheres, as in Proposition \ref{folinsphe}, and that when $\mathrm{ind}(x) \geq \mathrm{ind}(y)$ the foliation $\mathscr{L}^s_x$ refines the foliation $\mathscr{S}^s_y$.

\begin{thm}
\label{fitncpar}
For every stationary point $x$ there are:
\begin{enumerate}[(i)]
\item An invariant tubular neighborhood
\[
\pi_x^u : U_x \rightarrow W^u(x)
\]
of $W^u(x)$ with associated stable foliation
\[
\mathscr{L}^s_x := \{L^s_x(p\}_{p\in W^u(x)}.
\]
Here, $U_x$ is contained in $\phi( \R \times D_x(1))$.
\item A partially smooth foliation $\mathscr{S}^s_x$ of $U_x\setminus W^u(x)$ into $(n-k-1)$-dimensional spheres, where $k=\ind(x)$, that refines $\mathscr{L}^s_x$ and is $\mathscr{L}^s_x$-smooth. Each leaf of $\mathscr{S}^s_x$ is transverse to the unstable manifold of every stationary point of $\phi$. 
\item A one-parameter family $\{U_x(r)\}_{r\in (0,1]}$ of invariant open neighborhoods of $x$ such that 
\[
\begin{split}
U_x(1)=U_x, \quad & \overline{U_x(s)} \subset U_x(r) \cup \overline{W^u(x) \cup W^s(x)} \quad \mbox{ for all } 0<s<r\leq 1, \\ &
\bigcap_{r\in (0,1]} U_x(r) = W^u(x) \cup W^s(x),
\end{split}
\]
and for every $r\in (0,1]$ each leaf of $\mathscr{S}^s_x$ is either fully contained in or disjoint from $U_x(r)$.
\end{enumerate}
Moreover, these objects satisfy the following compatibility condition: 
\begin{enumerate}[(i)]
\setcounter{enumi}{3}
\item If $x$ and $y$ are stationary points with $\mathrm{ind}(x) \geq \mathrm{ind}(y)$, then the foliation $\mathscr{L}^s_x$ refines the foliation $\mathscr{S}^s_y$ and is $\mathscr{S}^s_y$-smooth.
\end{enumerate}
\end{thm}

In the case of two distinct stationary points $x$ and $y$ of the same index, the refinement property stated in (iv) implies that $U_x$ and $U_y$ are disjoint (this follows also from (MS-2), together with the fact that $U_x$ is contained in the flow evolution of $D_x(1)$). Together with (ii), the same property implies that  if $\mathrm{ind}(x) \geq \mathrm{ind}(y)$ then the foliation $\mathscr{L}^s_x$ refines the foliation $\mathscr{L}^s_y$ and is $\mathscr{L}^s_y$-smooth. 

The foliations appearing in Theorem \ref{fitncpar} are constructed inductively starting from the stationary points of index $0$. If $x$ is such a stationary point then $W^s(x)$ is an invariant open set, and we just set
\[
U_x = U_x(r)  := W^s(x) \quad \forall r\in (0,1], \qquad \pi^u_x(p):= x \quad \forall p\in W^s(x).
\] 
Then $(U_x,\pi_x^u)$ is an invariant tubular neighborhood of $W^u(x)=\{x\}$. The corresponding stable foliation has only the leaf, namely $W^s(x)$. The set $U_x\setminus W^u(x) = W^s(x)\setminus \{x\}$ has a smooth invariant foliation into embedded spheres that is obtained by choosing one smoothly embedded $(n-1)$-dimensional sphere $S^s_x\subset W^s(x)\setminus \{x\}$ that is transverse to the flow $\phi$ - for instance $S^s_x := \partial D^s_x(1)$ - and by letting it evolve by $\phi$. With this choices, the \\properties stated in (i), (ii) and (iii) hold trivially, wheras (iv) holds vacuously.

Arguing inductively, we now fix a natural number $0< k\leq n$, a set of rest points $Y$ which contains all the stationary points of index smaller than $k$ plus possibly some of index $k$, and we assume that for each $y\in Y$ we have defined the invariant tubular neighborhood $\pi^u_y : U_y \rightarrow W^u(y)$ with stable foliations $\mathscr{L}^s_y$, the foliations $\mathscr{S}^s_y$ and the one-parameter families of neighborhoods $\{U_y(r)\}_{r\in (0,1]}$ in such a way that the conditions stated in (i), (ii), (iii) and (iv) hold. We then fix a stationary point $x$ of index $k$ that is not in $Y$. Our aim is to prove the following result.

\begin{prop}
\label{laprop}
There are:
\begin{enumerate}[(i)]
\item An invariant tubular neighborhood
\[
\pi_x^u : U_x \rightarrow W^u(x)
\]
of $W^u(x)$ with associated stable foliation
\[
\mathscr{L}^s_x := \{L^s_x(p\}_{p\in W^u(x)}.
\]
Here, $U_x$ is contained in $\phi( \R \times D_x(1))$.
\item A partially smooth foliation $\mathscr{S}^s_x$ of $U_x\setminus W^u(x)$ into $(n-k-1)$-dimensional spheres, $k=\ind(x)$, that refines $\mathscr{L}^s_x$ and is $\mathscr{L}^s_x$-smooth. Each leaf of $\mathscr{S}^s_x$ is transverse to the unstable manifold of every stationary point of $\phi$. 

\item A one-parameter family $\{U_x(r)\}_{r\in (0,1]}$ of invariant open neighborhoods of $x$ such that 
\[
\begin{split}
U_x(1)=U_x, \quad & \overline{U_x(s)} \subset U_x(r) \cup \overline{W^u(x) \cup W^s(x)} \quad \mbox{ for all } 0<s<r\leq 1, \\ &
\bigcap_{r\in (0,1]} U_x(r) = W^u(x) \cup W^s(x),
\end{split}
\]
and for every $r\in (0,1]$ each leaf of $\mathscr{S}^s_x$ is either fully contained in or disjoint from $U_x(r)$.
\end{enumerate}
Moreover, these objects satisfy the following condition: 
\begin{enumerate}[(i)]
\setcounter{enumi}{3}
\item  For every $y\in Y$ the foliation $\mathscr{L}^s_x$ refines the restriction of the foliation $\mathscr{S}^s_y$ to $U_y(1/2)$ and is $\mathscr{S}^s_y|_{U_y(1/2)}$-smooth.
\end{enumerate}
\end{prop}

Theorem \ref{fitncpar} follows easily from successive applications of the above proposition. Indeed, at each step one replaces the invariant neighborhoods $U_y(r)$ by the smaller ones $U_y(r/2)$ for all $y\in Y$ and thus obtains statement (iv) of Theorem \ref{fitncpar} for all pairs of rest points in $Y\cup \{x\}$. The remaining part of this section is devoted to the proof of Proposition \ref{laprop}.

As in Section \ref{admsec}, we shall use the notation $L^s_y(p)$ in order to denote the leaf of the foliation $\mathscr{L}_y^s$ passing through $p\in U_y(1)$. This will cause no confusion with our choice of parametrizing the leaves of $\mathscr{L}_y^s= \{L^s_y(p)\}_{p\in W^u(y)}$ by points of $W^u(y)$, since the leaf $L^s_y(p)$ for $p\in W^u(y)$ does indeed pass through $p$. Similarly, the symbol $S^s_y(p)$ will denote the leaf of $\mathscr{S}_y^s$ passing through $p\in U_y(1)\setminus W^s(y)$. 

\begin{lem}
\label{smoothness}
For every pair of stationary points $y,y'\in Y$ the map 
\[
p \mapsto T_p S^s_y(p)
\]
is $\mathscr{L}_{y'}^s$-smooth.
\end{lem}

\begin{proof}
If $y'=y$ then the above map is $\mathscr{L}_{y'}^s$-smooth because $\mathscr{L}_{y'}^s=\mathscr{L}_{y}^s$ is a partially smooth foliation. A fortiori, this map is $\mathscr{L}_{y'}^s$-smooth whenever $\mathrm{ind}(y')\geq \mathrm{ind}(y)$, because in this case the foliation $\mathscr{L}_{y'}^s$ refines the foliation $\mathscr{L}_{y}^s$. Finally, assume that $\mathrm{ind}(y')< \mathrm{ind}(y)$. In this case, conditions (ii) and (iv) of Theorem \ref{fitncpar} imply that the foliation $\mathscr{L}^s_y$ refines $\mathscr{L}^s_{y'}$ and is $\mathscr{L}^s_{y'}$-smooth. This is equivalent to the fact that the map $p \mapsto T_p S^s_y(p)$ is $\mathscr{L}^s_{y'}$-smooth.
\end{proof}

In order to simplify the notation, we set
\[
\dot{W}^u(x) := W^u(x)\setminus \{x\}.
\]
Consider the smooth function
\[
\tau: \dot{W}^u(x) \rightarrow \R, \qquad \tau(p) :=  \inf \{ t\in \R \mid p\in \phi^t(D^u_x(1))\},
\]
which satisfies
\[
\phi^{-\tau(p)} (p) \in \partial D^u_x(1), \qquad \tau(\phi^t(p)) = \tau(p) + t \qquad \forall p\in  \dot{W}^u(x), \; \forall t\in \R.
\]
Then the map
\begin{equation}
\label{diffeo}
\dot{W}^u(x) \times \mathrm{int}(D^s_x(1)) \rightarrow M, \qquad (p,q) \mapsto \phi^{\tau(p)}( \phi^{-\tau(p)}(p),q),
\end{equation}
is a diffeomorphism onto an open invariant subset $\tilde{U}(1)$ of $M$. The pair in the argument of $\phi^{\tau(p)}$ refers to the identification of a neighborhood $D_x(1)$ of $x$ with $D^u_x(1)\times D^s_x(1)$. This map conjugates the restrictions of $\phi$ to $\tilde{U}_x(1)$ to the flow
\[
(t,p,q) \mapsto (\phi^t(p),q)
\]
on $\dot{W}^u(x) \times \mathrm{int}(D^s_x(1))$. For every $s\in (0,1]$, we denote by $\tilde{U}_x(s)$ the image of $\dot{W}^u(x) \times \mathrm{int}(D^s_x(s))$ by the map (\ref{diffeo}). By construction,
\[
\tilde{U}_x(s) \subset \tilde{U}_x(s') \quad \mbox{if } 0<s<s'\leq 1 \qquad \mbox{and} \quad \bigcap_{s\in (0,1]} \tilde{U}_x(s) = \dot{W}^u(x).
\]
By (MS-2), each point in $\tilde{U}_x(1)$ belongs to the stable manifold of some stationary point $y$ with $\ind (y) < \ind(x)=k$.

Denote by $Z$ the set of all stationary points of $\phi$ of Morse index less than $k$. Then $Z\subset Y$ and $Z\setminus Y$ consists of stationary points of index $k$ other than $x$. Thanks to (MS-2) and to the fact that $U_y$ is contained in $\phi(\R \times D_y(1))$ for every $y\in Y\setminus Z$, the set $\tilde{U}_x(1)$ is disjoint from each $U_y(1)$ with $y\in Y \setminus Z$. The invariant sets $U_x(r)$, $r\in (0,1]$, that we wish to construct will be contained in the set $\tilde{U}_x(1)\cup W^u(x)$, which is also disjoint from each $U_y(1)$ with $y\in Y \setminus Z$. Therefore, we will need to check property (iv) of Proposition \ref{laprop} only for stationary points $y$ belonging to $Z$.

\begin{lem}
\label{lemmulo}
If $s$ is small enough then there exists a continuous $\phi$-invariant Riemannian metric $g$ on $\tilde{U}_x(s)$ such that:
\begin{enumerate}[(i)]
\item For every $y\in Z$ and every $p\in \dot{W}^u(x) \cap U_y(3/4)$ the orthogonal complement of $T_p W^u(x)$ is contained in $T_p S^s_y(p)$.
\item For every $y\in Z$ the metric $g$ is $\mathscr{L}^s_y$-smooth.
\end{enumerate}
\end{lem}

\begin{proof}
Using the conjugacy (\ref{diffeo}), it is easy to construct a smooth $\phi$-invariant metric $\widehat{g}$ on $\tilde{U}_x(1)$, and we denote by $\widehat{P}_V$ the corresponding orthogonal projection onto a linear subspace $V\subset T_p M$, $p\in \tilde{U}_x(1)$. Using again the diffeomorphism (\ref{diffeo}), we extend the smooth $\phi$-invariant map 
\[ 
p\mapsto \widehat{P}_{T_p W^u(x)}, \quad p\in W^u(x),
\]
to a smooth $\phi$-invariant map $\widehat{P}_x$ on $\tilde{U}_x(1)$ such that $\widehat{P}_x(p)$ is a $\widehat{g}$-orthogonal projector on $T_p M$ for every $p\in \tilde{U}_x(1)$. We fix a $\phi$-invariant smooth partition of unity $\{\varphi_y\}_{y\in Z}$ on $\tilde{U}_x(1)$ such that for every $y\in Z$ we have
\begin{equation}
\label{cphi}
\supp \varphi_y \subset U_y(1) \setminus \bigcup_{\ind(y)< \ind (z)<k}U_z ( 3/4).
\end{equation}
A set of functions as above can be defined by choosing $\chi_y: U_y(1) \rightarrow [0,1]$, $y\in Z$, to be a smooth $\phi$-invariant function taking the value 1 on $U_y(3/4)$ and the value $0$ on $U_y(1)\setminus U_y(4/5)$, by extending it to a smooth function on $\tilde{U}_x(1)$ giving it the value 0 on $\tilde{U}_x(1)\setminus U_y(1)$, 
and by setting
\[
\varphi_y := \chi_y \cdot \prod_{\ind(y) < j < k} \Bigl( 1 - \sum_{\ind (z) = j} \chi_z \Bigr).
\]
For each $p\in \tilde{U}_x(1)$ we define the following linear endomorphism of $T_p M$
\[
S(p):= \widehat{P}_x(p) + \sum_{y\in Z} \varphi_y(p) \widehat{P}_{T_p S^s_y(p)}.
\]
Each $S(p)$ is self-adjoint and semi-positive with respect to the metric $\widehat{g}$ and, as a map, $S$ is $\mathscr{L}_z^s$-smooth for every $z\in Z$, because $\widehat{P}_x$ and the partition of unity are smooth and each map $p\mapsto \widehat{P}_{T_p S^s_y(p)}$ is $\mathscr{L}_z^s$-smooth thanks to Lemma \ref{smoothness}. If $p$ belongs to $\dot{W}^u(x)$ and $u$ is a vector in $T_p M$ and $\|\cdot\|$ is the norm that is induced by $\widehat{g}$, we have the identity
\[
\widehat{g}(S(p)u,u) = \Bigl\| \widehat{P}_{T_p W^u(x)} u \Bigr\|^2 + \sum_{y\in Z} \varphi_y(p) \Bigl\| \widehat{P}_{T_p S^s_y(p)} u \Bigr\|^2.
\]
Thanks to the fact that the spheres $S^s_y(p)$ are transverse to $W^u(x)$ (see statement (ii) in Theorem \ref{fitncpar}) and that $\varphi_y(p)>0$ for at least one $y\in Z$, the above identity implies that $S(p)$ is positive for every $p\in \dot{W}^u(x)$. Therefore, $S(p)$ is positive for every $p$ in a neighborhood of the compact set $\partial D^u_x(1)$ and, by the $\phi$-invariance, it is positive  for every $p$ in $\tilde{U}_x(s)$ provided that $s$ is small enough.

We fix such a small $s$ and define the metric $g$ on $\tilde{U}_x(s)$ as
\[
g(u,v) := \widehat{g}(S(p)^{-1}u,v) \qquad \forall u,v\in T_p M.
\]
The metric $g$ is $\phi$-invariant and satisfies (ii). Thanks to (\ref{cphi}), in the last sum in the definition of $S(p)$ only stationary points $y\in Z$ whose index is at least $\ind(z)$ may appear. Since the foliation $\mathscr{S}^s_y$ refines the foliation $\mathscr{S}^s_z$ for all rest points $z\in Z$ with $\mathrm{ind}(y) \geq \mathrm{ind}(z)$, we have the inclusion
\[
T_p S^s_y(p) \subset T_p S^s_z(p)
\]
for all $y\in Z$ such that $\varphi_y(p)>0$. We deduce that for such a stationary point $y$, the projector $\widehat{P}_{T_p S^s_y(p)}$ vanishes on the orthogonal complement $V$ of $T_p S^s_z(p)$ with respect to the metric $\widehat{g}$. The orthogonal complement of $T_p S^s_z(p)$ with respect to the metric $g$ is then the subspace 
\[
S(p) V = \widehat{P}_{T_p W^u(x)}  V,
\]
which is contained in $T_p W^u(x)$. By passing to the $g$-orthogonal spaces, we deduce that the $g$-orthogonal complement of $T_p W^u(x)$ is contained in $T_p S^s_z(p)$, as claimed in (i).
\end{proof}

Let $s\in (0,1]$ be small enough so that $\tilde{U}_x(s)$ carries a Riemannian metric $g$ with the properties stated in the above lemma.
The orthogonal projector onto a linear subspace $V\subset T_p M$, $p\in \tilde{U}_x(s)$, with respect to $g$ will be denoted by $P_V$.

By using the diffeomorphism (\ref{diffeo}), we can construct smooth $\phi$-invariant vector fields $\tilde{v}_1,\dots,\tilde{v}_{n-k}$ on $\tilde{U}_x(s)$ such that for each $p\in \dot{W}^u(x)$ the vectors $\tilde{v}_1(p),\dots,\tilde{v}_{n-k}(p)$ span the $g$-orthogonal complement of $T_p W^u(x)$ in $T_p M$.

Arguing as in the proof of Lemma \ref{lemmulo}, we can find a $\phi$-invariant smooth partition of unity $\{\varphi_y\}_{y\in Z}$ on $\tilde{U}_x(s)$ such that for each $y\in Z$
\begin{equation}
\label{cphi2}
\supp \varphi_y \subset U_y(3/4) \setminus \bigcup_{\ind(y)< \ind (z)<k}U_z ( 1/2).
\end{equation}
We define the $\phi$-invariant vector fields $v_1,\dots,v_{n-k}$ on $\tilde{U}_x(s)$ as
\[
v_j(p) := \sum_{y\in Z} \varphi_y(p) P_{T_p S^s_y(p)} \tilde{v}_j(p) \qquad \forall p\in \tilde{U}_x(s).
\]
Condition (\ref{cphi2}) and the fact that the foliations $\mathscr{S}^s_y$, $y\in Z$, refine each other imply that each vector field $v_j$ satisfies
\begin{equation}
\label{tang}
v_j(p) \in T_p S^s_y(p)\subset L^s_y(p) \qquad \forall p\in \tilde{U}_x(s) \cap U_y(1/2), \;\; \forall y\in Z.
\end{equation}
This fact and the regularity properties of the foliations $\mathscr{S}^u_y$ and of the metric $g$ that are stated in Lemmas \ref{smoothness} and \ref{lemmulo} imply that each $v_j$ is a partially smooth vector field with respect to the foliation $\mathscr{L}^s_y|_{U_y(1/2)}$, for every $y\in Z$. Finally, the fact that each $\varphi_y$ is supported in $U_y(3/4)$ and condition (i) of Lemma \ref{lemmulo} imply that for every $p\in \dot{W}^u(x)$ the operator
\[
\sum_{y\in Z} \varphi_y(p) P_{T_p S^s_y(p)}
\]
is the identity on the orthogonal complement of $T_p W^u(x)$ in $T_q M$ and thus the vectors $v_j(p) = \tilde{v_j}(p)$ form a basis of this orthogonal complement. 

Thanks to the discussion at the end of Section \ref{admsec}, each vector field $v_j$ has a well-defined continuous maximal local flow, which we denote by 
\[
\phi_{v_j} : \mathrm{dom} (\phi_{v_j}) \rightarrow \tilde{U}_x(s), \qquad \mathrm{dom} (\phi_{v_j}) \subset \R \times \tilde{U}_x(s),
\]
and which preserves the leaves of every foliation $\mathscr{S}^s_y|_{U_y(1/2)}$, $y\in Z$, and is $\R\times \mathscr{L}^s_y|_{U_y(1/2)}$-smooth. The fact that the vector fields $v_j$ are $\phi$-invariant implies that their flows commute with the flow $\phi$.

We can use the local flows $\phi_{v_j}$ to construct the following map $\psi$ on a subset $\mathrm{dom}(\psi)$ of $\R^{n-k} \times \dot{W}^u(x)$:
\[
\psi: \mathrm{dom}(\psi) \rightarrow M, \qquad 
\psi(t_1,\dots,t_{n-k},p) := \phi_{v_{n-k}}^{t_{n-k}} \circ \dots \circ \phi_{v_1}^{t_1} (p),
\]
Here, $\mathrm{dom}(\psi)$ is the maximal subset of $\R^{n-k} \times \dot{W}^u(x)$ on which the above composition is defined. By the local flow properties of the maps $\phi_{v_j}$, $\mathrm{dom}(\psi)$ is an open neighborhood of $\{0\} \times \dot{W}^u(x)$. Moreover, the fact that the flows $\phi_{v_j}$ commute with the flow $\phi$ imply that $\mathrm{dom}(\psi)$ is invariant with respect to the flow $\mathrm{id}_{\R^{n-k}} \times \phi$ and that $\psi$ conjugates the latter flow with $\phi$: 
\begin{equation}
\label{inva}
\phi^t\circ \psi(t_1,\dots,t_{n-k},p) = \psi (t_1,\dots,t_{n-k},\phi^t(p)) \qquad \forall t\in \R, \; (t_1,\dots,t_{n-k},p)\in \mathrm{dom}(\psi).
\end{equation}
The domain of $\psi$ is covered by the open sets
\[
\bigl\{\R^{n-k} \times \bigl(U_y(1/2) \cap \dot{W}^u(x)\bigr)\bigr\}_{y\in Z},
\]
each of which carries the partially smooth foliations
\[
\begin{split}
\widehat{\mathscr{S}}^s_y := \bigl\{ \widehat{S}^s_y(p,r) := \mathrm{dom}(\psi) \cap \bigl(  \R^{n-k} \times \bigl( S^s_y(p,r) \cap \dot{W}^u(x)\bigr) \bigr)\bigr\}_{(p,r) \in W^u(y)\times (0,1/2)} , \\
\widehat{\mathscr{L}}^s_y := \bigl\{ \widehat{L}^s_y(p) := \mathrm{dom}(\psi) \cap \bigl( \R^{n-k} \times \bigl( L^s_y(p) \cap U_y(1/2)\cap  \dot{W}^u(x)\bigr)\bigr)\bigr\}_{p \in W^u(y)}.
 \end{split}
\]
Since each $\phi_{v_j}$ preserves the leaves of $\mathscr{S}^s_y$ for every $y\in Z$, the map $\psi$ satisfies
\begin{equation}
\label{raffa}
\psi \bigl(\widehat{S}^s_y(p,r) \bigr)  \subset S^s_y(p,r) \qquad \forall (p,r) \in W^u(y)\times (0,1/2), \; \forall y\in Z.
\end{equation}
Moreover, the smoothness properties of the flows $\phi_{v_j}$ and the smoothness of compositions that we discussed in Section \ref{admsec} implies that $\psi$ is smooth with respect to all the above foliations. In particular, $\psi$ is infinitely differentiable in $(t_1,\dots,t_{n_k})\in \R^{n-k}$ and the differentials of every order in these variables are continuous on $\mathrm{dom}(\psi)$. Differentiation at points in $\{0\} \times \dot{W}^u(x)$ yields
\begin{equation}
\label{diffe}
\partial_{t_j} \psi(0,p) = v_j(p) \qquad \forall j\in \{1,\dots,n-k\}, \; p\in \dot{W}^u(x).
\end{equation}

\begin{lem} 
There is an open neighborhood $\Omega$ of $\{0\}\times \dot{W}^u(x)$ in $\mathrm{dom}(\psi)$ that is invariant with respect to the flow $\mathrm{id}_{\R^{n-k}} \times \phi$ and on which $\psi$ is a homeomorphism. 
\end{lem}

\begin{proof}
We first prove that $\psi$ is a local homeomorphism at every point of the form $(0,p)\in \R^{n-k} \times \dot{W}^u(x)$. Let $p\in \dot{W}^u(x)$. Then $p$ belongs to $U_y(1/2)$ for a suitable stationary point $y\in Z$, and hence $(0,p)$ belongs to the domain of the foliation $\widehat{\mathscr{S}}^s_y$. As discussed above, $\psi$ maps each leaf of $\widehat{\mathscr{S}}^s_y$ into a leaf of $\mathscr{S}^s_y$ and is smooth on each leaf, with all leafwise differentials varying continuously on the domain of the foliation $\widehat{\mathscr{S}}^s_y$ . Moreover, (\ref{diffe}) and the fact that $\psi(0,\cdot)$ is the identity on $\dot{W}^u(x)$ imply that the differential at $(0,p)$ of the restriction of $\psi$ to the leaf of $\widehat{\mathscr{S}}^s_y$ through $(0,p)$ is an isomorphism onto the tangent space to the leaf of $\mathscr{S}^s_y$ through $p$. Then the parametric inverse mapping theorem implies that $\psi$ is a local homeomorphism at $(0,p)$, as claimed.

Our next claim is that the compact set $\{0\} \times \partial D^u_x(1)$ has an open neighborhood $\Omega'$ in $\mathrm{dom}(\psi)$ such that the restriction of $\psi$ to $\Omega'$ is a homeomorphism onto an open subset of $M$. In order to prove this fact, we fix an open neighborhood $V$ of $\partial D^u_x(1)$ in $W^u(x)$ with $\overline{V}\subset \dot{W}^u(x)$, and we look at open sets $\Omega'$ of the form
\[
\Omega' = B_{\epsilon}^{n-k} \times V,
\]
where $B_{\epsilon}^{n-k}$ denotes the open ball of radius $\epsilon$ around $0$ in $\R^{n-k}$. Such a set $\Omega'$ is contained in $\mathrm{dom}(\psi)$ for $\epsilon$ small enough. Moreover, the fact that $\psi$ is a local homeomorphism at the points in $\{0\} \times \dot{W}^u(x)$ implies that $\psi$ restricts to an open map on $\Omega'$ for $\epsilon$ small enough. Therefore, it suffices to show that $\psi$ is injective on $\Omega'$ if $\epsilon$ is small enough. This is an easy consequence of the fact that $\psi$ is a local homeomorphism at the points in $\{0\}\times \dot{W}^u(x)$, on which $\psi$ is injective. Indeed, assume by contradiction that $\psi$ is not injective on $\Omega'$ for any positive $\epsilon$. Then we can find sequences $(t_j,p_j)$ and $(t_j',p_j')$ in $\R^{n-k} \times V$ such that $t_j\rightarrow 0$, $t_j'\rightarrow 0$, $(t_j,p_j) \neq (t_j',p_j')$ and $\psi(t_j,p_j) = \psi(t_j',p_j')$. Up to passing to suitable subsequences, we may assume that $p_j\rightarrow p$ and $p_j'\rightarrow p'$ for some $p,p'\in \overline{V} \subset \dot{W}^u(x)$. Then
\[
p = \psi(0,p) = \lim_{j\rightarrow \infty} \psi(t_j,p_j) = \lim_{j\rightarrow \infty} \psi(t_j',p_j') = \psi(0,p') = p'.
\]
The presence of the sequences $(t_j,p_j) \neq (t_j',p_j')$ both converging to $(0,p)=(0,p')$ and satisfying $\psi(t_j,p_j) = \psi(t_j',p_j')$ contradicts the injectivity of $\psi$ in a neighborhood of $(0,p)=(0,p')$. 

The desired open neighborhood $\Omega$ of $\{0\}\times \dot{W}^u(x)$ is obtained by applying the flow $\mathrm{id}_{\R^{n-k}} \times \phi$ to $\Omega'$. The resulting set $\Omega$ is invariant under this flow, and $\psi$ is a homeomorphism on it because it is a homeomorphism on $\Omega'$ and thanks to (\ref{inva}).
\end{proof}

By (\ref{diffe}), the differential $d_1 \psi(0,p)$ with respect to $(t_1,\dots,t_{n-k})\in \R^{n-k}$ is injective for every $p\in \dot{W}^u(x)$. Up to replacing the open set $\Omega$ by a smaller one, we may assume that $d_1 \psi(t,p)$ is injective for every $(t,p)\in \Omega$. After identifying $\dot{W}^u(x)$ with $\R^k \setminus \{0\}$ by means of a smooth diffeomorphism, we conclude that $\psi$ is a partially smooth parametrization of the foliation 
\[
\tilde{\mathscr{L}}^s_x := \Bigl\{ \tilde{L}^s_x(p) := \psi\bigl(\Omega \cap \bigr(\R^{n-k} \times \{p\}\bigr) \bigr) \Bigr\}_{p\in \dot{W}^u(x)},
\]
which is then a $(n-k)$-dimensional partially smooth foliation of the open set
\[
\tilde{U}_x' := \psi(\Omega).
\]
Condition (\ref{inva}) guarantees that $\tilde{U}_x'$ is $\phi$-invariant and so is the foliation $\tilde{\mathscr{L}}^s_x$. Condition (\ref{raffa}) implies that $\tilde{\mathscr{L}}^s_x$ refines the foliations $\mathscr{S}_y^s|_{U_y(1/2)}$ for every $y\in Z$. By the regularity properties of $\psi$, the foliation $\tilde{\mathscr{L}}^s_x$ is $\mathscr{L}_y^s|_{U_y(1/2)}$ for every $y\in Z$. Finally, (\ref{diffe}) implies that each leaf $\tilde{L}^s_x(p)$ is transverse to $W^u(x)$ at their unique intersection point $p$.

By Proposition \ref{stafol-prop} the invariant set
\[
U'_x := \tilde{U}'_x \cup W^u(x)
\]
is open and the foliation $\tilde{\mathscr{L}}^s_x$ can be completed by adding the leaf $W^u(x)$ to the stable foliation $\mathscr{L}^s_x$ of a tubular neighborhood
\[
\pi^u_x :  U'_x \rightarrow W^u(x).
\]
By Proposition \ref{folinsphe} there is an open invariant set $U_x \subset U'_x$ and a partially smooth invariant foliation $\mathscr{S}^s_x$ of $U_x \setminus W^u(x)$ into $(n-k-1)$-dimensional spheres that refines the foliation $\mathscr{L}^s_x$ and is $\mathscr{L}^s_x$-smooth. 

By the Morse-Smale assumption, the leaves of $\mathscr{S}^s_x$ that are contained in $W^s(x)$ are transverse to the unstable manifold of every stationary point of $\phi$. This transversality is actually uniform, as shown in Corollary \ref{uniftra}. Therefore,
up to reducing $U_x$ even more, we can assume that all the leaves of $\mathscr{S}^s_x$ are transverse to the unstable manifold of every stationary point.

We rename $\pi^u_x$ to be the restriction of the same map to $U_x$ and $\mathscr{L}^s_x$ to be the restriction of the associated stable foliation to $U_x$. With these choices, the conditions stated in Proposition \ref{laprop} (i), (ii) and (iv) hold. Proposition \ref{folinsphe} also gives us the existence of a family of invariant open neighborhoods $U_x(r)$, $0<r\leq 1$, such that statement (iii) in Proposition \ref{laprop} holds. This concludes the proof of this proposition, and hence of Theorem \ref{fitncpar}.

\medskip

In Example \ref{cubic} we have seen that there might be curves that are tangent to the leaves of a partially smooth foliation but contained in none of them. The next results says that this phenomenon does not happen for the foliations that are constructed in Theorem \ref{fitncpar}. 

\begin{prop}
\label{uniqueness}
The foliations $\mathscr{L}^s_x$  have the following property: each continuously differentiable curve in $U_x(1)$ that is tangent to the leaves of $\mathscr{L}^s_x$ is fully contained in one leaf. The analogous fact holds for the foliations $\mathscr{S}^s_x$.
\end{prop}

\begin{proof}
Let $I$ be an interval and $\gamma:I \rightarrow U_x(1)$ be a continuously differentiable map such that $\dot{\gamma}(t)$ belongs to the tangent space to the leaf of $\mathscr{L}^s_x$ through $\gamma(t)$ for every $t\in I$.
Consider the smooth stratification of $M$ given by the stable manifolds:
\[
W^s_j := \bigcup_{\substack{z\in \mathrm{stat}(\phi) \\ \ind(z)=n-j}} W^s(z), \quad 0\leq j \leq n, \qquad \dim W^s_j = j.
\]
Let $h$ be the largest integer such that $\gamma(I)\cap W^s_h \neq \emptyset$ and notice that $h\geq n- \ind(x)$ because 
$U_x(1)$ is covered by the stable manifolds of stationary points of index at most $ \ind(x)$. Then $\gamma$ is a curve in
\[
\bigcup_{0\leq j \leq h} W_j^s,
\]
and since $W_h^s$ is open in the latter set, the set of $t\in I$ such that $\gamma(t)\in W_h^s$ is a non-empty open subset of $I$. Denote by $J\subset I$ a maximal interval in this set. Then $\gamma(J)$ is contained in one connected component of $W^s_h$, that is, in the stable manifold of some stationary point $y$ with $\mathrm{ind}(y)=n-h \leq \mathrm{ind}(x)$.

By claim (iv) in Theorem \ref{fitncpar}, $\mathscr{L}^s_x$ refines the foliation $\mathscr{L}^s_y$ and is $\mathscr{L}^s_y$-smooth. In particular, $\mathscr{L}^s_x$ restricts to a smooth foliation of $W^s(y)\cap U_x(1)$. Theorefore, the curve $\gamma(J)\subset W^s(y)\cap U_x(1)$, which is tangent to the foliation $\mathscr{L}^s_x$, is actually contained in a leaf $L^s_x(p)$ of this foliation. The fact that each leaf of $\mathscr{L}^s_x$ is relatively closed in $U_x(1)$ implies that $J$ has no boundary points in $I$. Since $I$ is connected, we have $I=J$ and $\gamma(I)\subset L^s_x(p)$.

The analogous claim for $\mathscr{S}^s_x$ easily follows. Indeed, any curve tangent to $\mathscr{S}^s_x$ is in particular tangent to $\mathscr{L}^s_x$, so by what we have shown above it is contained in a leaf of $\mathscr{L}^s_x$. Being $\mathscr{L}^s_x$-smooth, the foliation $\mathscr{S}^s_x$ restricts to a smooth foliation of every leaf of $\mathscr{L}^s_x$. Therefore, the curve is contained in a leaf of $\mathscr{S}^s_x$.
\end{proof}

Theorem \ref{introthm2} in the Introduction follows immediately from Theorem \ref{fitncpar} and the above proposition.

\begin{rem}
In Section \ref{admsec}, we have seen that Cauchy problems induced by partially smooth vector fields might not have uniqueness. However, the Cauchy problems for the vector fields $v_j$ that we have constructed in this section do have unique solutions, albeit they need not be Lipschitz-continuous. This follows easily from the fact that these vector fields are tangent to the leaves of $\mathscr{S}^s_y|_{U_y(1/2)}$, on which they are smooth, and from the property of these foliations that is established in Proposition \ref{uniqueness}.
\end{rem}

\begin{rem}
\label{smooth?}
In what follows, it would be enough to have a version of Theorem \ref{introthm2} with lower regularity on the foliations. Indeed, it would be enough to require the foliations to have $C^1$ leaves with tangent spaces varying continuously, and such that the restriction of  each foliation of dimension $k$ to a leaf of a foliation of higher dimension is a genuine $C^1$ foliation. It is conceivable that compatible foliations with this lower regularity can be constructed also assuming the flow to be of class $C^1$. However, our approach would still need some high regularity on the flow, namely a flow of class $C^{n-1}$, where $n=\dim M$. The reason is that our foliations are constructed by integrating vector fields that are obtained from smooth ones by projection on the leaves that have already been built. Therefore, in the inductive construction the vector fields we obtain have one degree of regularity less than the leaves that have already been built. 
\end{rem}

\section{The flow near the boundary of an unstable manifold} 
\label{nearbdrysec}

Let $x$ be a stationary point of $\phi$ of Morse index $k$. The aim of this section is to construct a positively complete local flow $\theta$
near on the complement of a compact set in $W^u(x)$ all of whose forward orbits $\theta^t(p)$ converge to some point in $\omega_{\theta}(p)$ in $\overline{W^u(x)} \setminus W^u(x)$, for a suitable continuous map $p\mapsto \omega_{\theta}(p)$. This flow will be induced by a continuous vector field $Y$ on $\dot{W}^u(x):= W^u(x) \setminus \{x\}$.

The points in $\dot{W}^u(x)$ belong to stable manifolds of stationary points of $\phi$ of index less than $k$. Therefore, $\dot{W}^u(x)$ is covered by the $\phi$-invariant open sets
\[
U_i := \dot{W}^u(x) \cap \bigcup_{\substack{y\in \mathrm{stat}(\phi) \\ \ind(y)=i}} U_y, \qquad i=0,\dots,k-1,
\]
where $U_y$ is the tubular neighborhood of the unstable manifold $W^u(y)$ that is given by Theorem \ref{introthm2} in the introduction. By the same theorem, each $U_i$ carries the $(k-i)$-dimensional partially smooth foliation
\[
\mathscr{L}^s_i := \bigcup_{\substack{y\in \mathrm{stat}(\phi) \\ \ind(y)=i}} \{L^s_y(p) \cap W^u(x)\}_{p\in W^u(y)}.
\]
When $0\leq i < j <k$, the foliation $\mathscr{L}^s_j$ refines the foliation $\mathscr{L}^s_i$ and is $\mathscr{L}^s_i$-smooth. Moreover, each foliation $\mathscr{L}_i^s$ is refined by the following $(k-i-1)$-dimensional partially smooth foliation 
\[
\mathscr{S}^s_i := \bigcup_{\substack{y\in \mathrm{stat}(\phi) \\ \ind(y)=i}} \{S^s_y(p,a) \cap W^u(x)\}_{(p,a)\in \Lambda_y}, \quad \Lambda_y \subset W^u(y) \times \R,
\]
of $U_i$, which is $\mathscr{L}_i^s$-smooth. The leaves of these foliations are denoted by 
\[
L^s_i(p) \mbox{ for } p\in \bigcup_{\ind (y)=i} W^u(y), \quad \mbox{and} \quad S^s_i(p,a) \mbox{ for } (p,a)\in \bigcup_{\ind (y)=i} \Lambda_y \subset \bigcup_{\ind (y)=i} W^u(y) \times \R.
\]
If $0\leq i < j <k$ then the foliation $\mathscr{L}^s_j$ refines the foliation $\mathscr{S}^s_i$. 

Furthermore, Proposition \ref{coerfu}, together with Remark \ref{fund}, implies that there are continuous functions
\[
\rho_i: U_i \rightarrow (0,+\infty)
\]
with the following properties:
\begin{enumerate}[($\rho$-1)]
\item If $p$ belongs to the unstable manifold of some stationary point $y$ of index $i$ and $(p_h)\subset U_i$ converges to $p$, then $\rho_i(p_h)\rightarrow 0$; moreover, for every neighborhood $N$ of $p$ in $M$ there exists an open neighborhood $P$ of $p$ in $W^u(y)$ and a positive number $\epsilon>0$ such that the open subset of $W^u(x)$ given by
\[
 \rho_i^{-1}((0,\epsilon))  \cap \bigcup_{p'\in P} L^s_i(p') 
\]
is contained in $N$.
\item If the sequence $(p_h)\subset U_i$ converges to a point in the relative boundary of $U_i$ in $W^u(x)$ then $\rho_i(p_h) \rightarrow +\infty$.
\item $\rho_i$ is constant on each leaf of $\mathscr{S}^s_i$.
\item $\rho_i$ is $\mathscr{L}^s_j$-smooth for every $j\in \{0,\dots,k-1\}$.
\item The differential of the restriction of $\rho_i$ to each leaf of $\mathscr{L}^s_i$ never vanishes.
\item $\rho_i$ is  differentiable along the flow $\phi$ with
\[
\frac{d}{dt} \rho_i \circ \phi^t \Bigr|_{t=0} = - \rho_i.
\]
\end{enumerate}

It will be convenient to consider the functions $\rho_i$ as defined on the whole of $W^u(x)$, by setting $\rho_i=+\infty$ outside of $U_i$. Thanks to ($\rho$-2), these $(0,+\infty]$-valued functions are continuous on $W^u(x)$.

The vector field $Y$ will be constructed by patching the vector fields on $U_i$ that are constructed in the next proposition.

\begin{lem} 
\label{locdata}
There are continuous vector fields $Y_i$ on $U_i$ that are tangent to $W^u(x)$ and have the following properties:
\begin{enumerate}[(i)]
\item The vector field $Y_i$ is tangent to the foliation $\mathscr{L}^s_i$ and is $\mathscr{L}^s_j$-smooth for every $j\in \{0,\dots,k-1\}$.
\item $d\rho_i[Y_i] = -\rho_i$ on $U_i$.
\item If $i<j$ then $d\rho_i[Y_j] = 0$ on $U_i \cap U_j$.
\end{enumerate}
\end{lem}

\begin{proof}
We fix a smooth metric on $W^u(x)$. This metric induces a smooth metric on each leaf of $\mathscr{L}^s_i$ and we denote by $\nabla \rho_i$ the leafwise gradient of the function $\rho_i$ with respect to it. The vector field $\nabla \rho_i$ is tangent to the leaves of $\mathscr{L}^s_i$ and, thanks to ($\rho$-3) and ($\rho$-5), it  never vanishes and is transverse to the leaves of $\mathscr{S}_i^s$. We define the vector field $Y_i$ to be the unique vector field that is parallel to $\nabla \rho_i$ and is normalized so that
\[
d\rho_i[Y_i] = - \rho_i.
\]
This vector field satisfies (i) thanks to the regularity property ($\rho$-4) of the function $\rho_i$ and to the smoothness of the metric. Property (ii) holds by construction. If $i<j$ then the foliation $\mathscr{L}^s_j$ refines the foliation $\mathscr{S}^s_i$. Together with the fact that $\rho_i$ is constant on the leaves of $\mathscr{S}^s_i$ - by ($\rho$-3) - and $Y_j$ is tangent to the leaves of $\mathscr{L}^s_j$, this implies (iii).
\end{proof}

We denote by $V_i$ the following open subset of $\dot{W}^u(x)$:
\[
V_i := \{p\in U_i \mid \rho_i(p) < 1\}.
\]
By ($\rho$-2) we have
\begin{equation}
\label{Vjei}
\overline{V_i} \cap \dot{W}^u(x) = \overline{V_i} \cap W^u(x) = \{  p\in U_i \mid \rho_i(p) \leq 1\}.
\end{equation}
Let $\chi_i: \dot{W}^u(x) \rightarrow [0,1]$, $0\leq i \leq k-1$, be smooth functions such that:
\begin{enumerate}[($\chi$-1)]
\item $\chi_i$ is supported in $U_i$.
\item For every $i<j$, $\chi_i=0$ on $\overline{V_j}\cap \dot{W}^u(x)$.
\item $\chi_i>0$ on $(\overline{V_i} \setminus \displaystyle \bigcup_{j>i} \overline{V_j})\cap \dot{W}^u(x)$.
\end{enumerate}
Such a family of functions is easy to construct: For every $i\in \{0,\dots,k-1\}$ choose a smooth function $\psi_i:\dot{W}^u(x) \rightarrow [0,1]$ such that
\[
\mathrm{supp} \, \psi_i \subset U_i, \qquad \psi_i^{-1}(\{1\}) = \overline{V_i} \cap W^u(x),
\]
and define
\[
\chi_i:= \psi_i \prod_{j>i} (1-\psi_i), \qquad \forall i\in \{0,\dots,k\}.
\]
The functions $\psi_i$ exist because $\overline{V_i}\cap W^u(x)$ and the complement of $ U_i$ in $\dot{W}^u(x)$ are disjoint closed subsets of $\dot{W}^u(x)$, thanks to (\ref{Vjei}), and because every closed subsets of a manifold is the set of zeroes of a smooth function. 

We now consider the following tangent vector field on $\dot{W}^u(x)$:
\[
Y:= \sum_{i=0}^{k-1} \chi_i Y_i.
\]
This vector field is continuous and $\mathscr{L}_i^s$-smooth for every $i\in \{0,\dots,k-1\}$. It need not be Lipschitz-continuous, but its Cauchy problems do have unique local solutions, as we show in the next lemma. 

\begin{lem}
The vector field $Y$ is uniquely integrable on $\dot{W}^u(x)$.
\end{lem}

\begin{proof}
We wish to prove that for every $u\in \dot{W}^u(x)$ the Cauchy problem
\[
\left\{ \begin{array}{l} \gamma'(t) = Y(\gamma(t)) \\ \gamma(0)=u \end{array} \right.
\]
has a unique local solution. Being $\mathscr{L}^s_0$-smooth, the vector field $Y$ is smooth on the open set $U_0$, so we do have local uniqueness when $u\in U_0$. The complement of $U_0$ in $\dot{W}^u(x)$ is covered by the open sets
\begin{equation}
\label{theset}
U_i \setminus \bigcup_{j<i} \mathrm{supp}\, \chi_j,
\end{equation}
for $i\in \{1,\dots,k-1\}$. On the above set, the vector field $Y$ has the form
\[
Y = \sum_{j\geq i} \chi_j Y_j.
\]
Since $Y_j$ is tangent to the foliation $\mathscr{L}^s_j$, which refines the foliation $\mathscr{L}^s_i$ for every $j>i$, the vector field $Y$ is tangent to the leaves of $\mathscr{L}^s_i$ on the set (\ref{theset}). Therefore, all local solutions of the above Cauchy problem with initial condition $u$ in the set (\ref{theset}) are tangent to the leaves of $\mathscr{L}^s_i$. By Proposition statement (iii) in Theorem \ref{introthm2}, these local solutions must be contained in the leaf of $\mathscr{L}^s_i$ through $u$. From the fact that the restriction of $Y$ to this leaf is smooth, we deduce that the local solution is unique.
\end{proof}

Being continuous and uniquely integrable, $Y$ defines a continuous local flow on $\dot{W}^u(x)$, that we denote by 
\[
\theta: \mathrm{dom} (\theta) \rightarrow \dot{W}^u(x), \qquad (t,p) \mapsto \theta^t(p),
\]
where
\[
\mathrm{dom} (\theta) \subset \R \times \dot{W}^u(x)
\]
is the maximal domain of existence.

\begin{lem}
\label{olem}
For every $p\in \overline{V_i}$ we have
\begin{equation}
\label{laines}
\frac{d}{dt} \rho_i \circ \theta^t(p) \Bigr|_{t=0} = - \chi_i(p) \rho_i(p) \leq 0.
\end{equation}
In particular, $\overline{V_i}$ is positively invariant under $\theta$.
\end{lem}

\begin{proof}
By ($\chi$-2), $\chi_j(p)=0$ for every $j<i$ and hence
\[
\frac{d}{dt} \rho_i \circ \theta^t(p) \Bigr|_{t=0} = d\rho_i(p)[Y(p)] = \sum_{j\geq i} \chi_j(p) d\rho_i(p)[Y_j(p)].
\]
By Lemma \ref{locdata} (ii) and (iii), the latter quantity equals
\[
- \chi_i(p) \rho_i(p),
\]
proving (\ref{laines}). The inequality in (\ref{laines}) and the fact that $\overline{V_i}$ is a sublevel of the function $\rho_i: W^u(x) \rightarrow (0,+\infty]$ imply that $\overline{V_i}$ is positively invariant under $\theta$.
\end{proof}

We now introduce the following open subset of $\dot{W}^u(x)$:
\[
V := \bigcup_{i=0}^{k-1} V_i.
\]
Notice that $V$ is a ``neighborhood of the boundary of $W^u(x)$'', meaning that the difference $W^u(x) \setminus V$ is compact.

\begin{prop}
\label{converge}
The forward orbit by $\theta$ of each $p\in \overline{V} \cap W^u(x)$ is defined for every $t\in [0,+\infty)$. Moreover, for every such $p$ there exists a point $\omega_{\theta}(p)\in \overline{W^u(x)} \setminus W^u(x)$ such that
\[
\lim_{t\rightarrow +\infty} \theta^t(p) = \omega_{\theta}(p),
\]
and the map
\[
\omega_{\theta} : \overline{V}\cap W^u(x) \rightarrow \overline{W^u(x)} \setminus W^u(x)
\]
is continuous.
\end{prop}

\begin{proof}
Let 
\[
\gamma: [0,T) \rightarrow \dot{W}^u(x), \qquad \gamma(t) := \theta^t(p),
\]
be the forward orbit of some $p\in \overline{V}\cap W^u(x)$, where $T\in (0,+\infty]$ is the maximal existence time.
Let $i\in \{0,\dots,k-1\}$ be the maximal index such that the forward orbit $\gamma$ meets $\overline{V_i}$: There is a number $t_0\in [0,T)$ such that $\gamma(t_0)\in \overline{V_i}$, and $\gamma(t)\notin \overline{V_j}$ for every $j>i$ and every $t\in [0,T)$. 

By Lemma \ref{olem}
\begin{equation}
\label{scende}
\gamma(t)\in \overline{V_i} \qquad \mbox{and} \qquad \frac{d}{dt} \rho_i \circ \gamma(t) = - \chi_i(\gamma(t)) \rho_i(\gamma(t)) \qquad \forall t\in [t_0,T).
\end{equation}
Since $\chi_i\leq 1$, we have
\[
\frac{d}{dt} \rho_i \circ \gamma \geq - \rho_i\circ \gamma \qquad \mbox{on } [t_0,T),
\]
and hence
\begin{equation}
\label{dalbasso}
\rho_i \circ \gamma(t) \geq \rho_i(\gamma(t_0))e^{t_0-t} \qquad \forall t\in [t_0,T).
\end{equation}
If $j>i$ we have
\begin{equation}
\label{dalbasso2}
\rho_j \circ \gamma (t) > 1 \qquad \forall t\in [0,T),
\end{equation}
because the forward orbit $\gamma$ does not meet $\overline{V_j}$. If $j<i$ we have, using again Lemma \ref{olem} and the fact that $\chi_j$ vanishes on $\overline{V_i}$ by ($\chi$-2),
\begin{equation}
\label{dalbasso3}
\rho_j \circ \gamma(t) = \rho_j (\gamma(t_0)) \in (0,+\infty] \quad \forall t\in [t_0,T).
\end{equation}
By (\ref{dalbasso2}) and (\ref{dalbasso3}), $\rho_j\circ \gamma$ is bounded away from zero on $[t_0,T)$ for every $j\neq i$. If, by contradiction, $T$ is finite then (\ref{dalbasso}) implies that also $\rho_i\circ \gamma$ is bounded away from zero on $[t_0,T)$. But then $\gamma([t_0,T))$ is contained in a compact subset of $\dot{W}^u(x)$ and this contradicts the fact that the maximal existence time $T$ is finite. We conclude that $T=+\infty$.

We now claim that the positive function $\rho_i\circ \gamma$, which by (\ref{scende}) is monotonically decreasing on $[t_0,+\infty)$, converges to zero for $t\rightarrow +\infty$. If this is not the case, then this function has a positive limit 
\[
r:= \inf_{[t_0,+\infty)} \rho_i\circ \gamma
\]
for $t\rightarrow +\infty$. Using again (\ref{dalbasso2}) and (\ref{dalbasso3}), we deduce that $\gamma([t_0,+\infty))$ is contained in a compact subset of $\dot{W}^u(x)$ and we find a sequence $(t_h)\subset [t_0,+\infty)$ with $t_h \rightarrow +\infty$ such that $\gamma(t_h)$ converges to some point $q\in \dot{W}^u(x)$, which belongs to $\overline{V_i}$ and satisfies $\rho_i(q)=r$. If $q$ is not in the set $\bigcup_{j>i} \overline{V_j}$, then Lemma \ref{olem} and condition ($\chi$-3) imply that
\[
\frac{d}{dt} \rho_i \circ \theta^t(q) \Bigr|_{t=0} = - \chi_i(q) \rho_i(q) <0,
\]
and hence for some small $\overline{t}>0$ we have
\[
\rho_i (\theta^{\overline{t}}(q)) < \rho_i(q) = r.
\]
By the continuity of $\theta$ and $\rho_i$, we have
\[
\rho_i ( \gamma(t_h+\overline{t}) ) = \rho_i( \theta^{\overline{t}}(\gamma(t_h)))  \rightarrow \rho_i (\theta^{\overline{t}}(q)) < r,
\]
and hence $\rho_i ( \gamma(t_h+\overline{t}) )<r$ for $h$ large enough, contradicting the assumption that $r$ is the infimum of $\rho_i$ over $\gamma([t_0,+\infty))$. Therefore, $q$ belongs to the set $\bigcup_{j>i} \overline{V_j}$.

Let $j>i$ be the largest index such that $q$ belongs to $\overline{V_j}$. By Lemma \ref{olem} and ($\chi$-3), we have
\[
\frac{d}{dt} \rho_j \circ \theta^t(q) \Bigr|_{t=0} = - \chi_j(q) \rho_j(q) <0.
\]
Therefore, for a small $\overline{t}>0$ we have
\[
\rho_j (\theta^{\overline{t}}(q)) < \rho_j(q) \leq 1,
\]
meaning that $\theta^{\overline{t}}(q)$ belongs to open set $V_j$. Therefore, the sequence
\[
\gamma(t_h+\overline{t}) = \theta^{\overline{t}}(\gamma(t_h)),
\]
which converges to $\theta^{\overline{t}}(q)$, eventually enters $V_j$. This contradicts the initial assumption that the forward orbit $\gamma$ does not meet any $\overline{V_j}$ with $j>i$.

We now know that $\gamma$ is defined on the whole interval $[0,+\infty)$, that $\gamma(t)$ belongs to $\overline{V_i}$ for all $t\geq t_0$ and that $\rho_i\circ \gamma(t)$ tends to zero for $t\rightarrow +\infty$. Using again property ($\chi$-2), we obtain the identity
\[
Y(\gamma(t)) = \sum_{j\geq i} \chi_j(\gamma(t)) Y_j(\gamma(t)) \qquad \forall t\geq t_0.
\]
By property (i) in Lemma \ref{locdata}, $Y_j$ is tangent to the foliation $\mathscr{L}^s_j$ for every $j$. Since the foliations $\mathscr{L}^s_j$ with $j\geq i$ refine the foliation $\mathscr{L}^s_i$, the above formula shows that the vectors $Y(\gamma(t))$ are tangent to the foliation $\mathscr{L}^s_i$ for every $t\geq t_0$. By Proposition \ref{uniqueness}, $\gamma([t_0,+\infty))$ is contained in a single leaf $L^s_i(p_{\infty})$, for some $p_{\infty}\in W^u(y)$, where $y$ is a stationary point of $\phi$ with $\ind(y)=i$. Thanks to ($\rho$-1), the fact that $\rho_i \circ \gamma(t)$ tends to $0$ for $t\rightarrow +\infty$ implies that $\gamma(t)$ converges to $p_{\infty}$ for $t\rightarrow +\infty$.

This proves that the forward orbit of each $p\in \overline{V}\cap W^u(x)$ converges to some point $\omega_{\theta}(p)$ in $\overline{W^u(x)} \setminus W^u(x)$. Now we wish to prove that the map
\[
\omega_{\theta} : \overline{V}\cap W^u(x) \rightarrow \overline{W^u(x)} \setminus W^u(x)
\]
is continuous. Let $p\in \overline{V}\cap W^u(x)$ and let $N$ be a neighborhood of $\omega_{\theta}(p)$ in $M$. By our previous considerations, $p_{\infty}:=\omega_{\theta}(u)$ is a point in the unstable manifold $W^u(y)$ of some stationary point $y$ of index $i$, $0\leq i \leq k-1$, and there exists $t_0$ such that 
\[
\theta^t(p) \in L^s_i(p_{\infty}) \qquad \forall t\geq t_0, \qquad \lim_{t\rightarrow +\infty} \rho_i(\theta^t(p)) = 0.
\]
Here $i$ is the maximal index such that the forward orbit of $p$ meets $\overline{V_i}$.
By ($\rho$-1) we can find an open neighborhood $P$ of $p_{\infty}$ in $W^u(y)$ and a positive number $\epsilon\in (0,1)$ such that 
\[
N_0:= \rho_i^{-1}((0,\epsilon))  \cap \bigcup_{p'_{\infty}\in P} L^s_i(p'_{\infty}), 
\]
which is an open subset of $W^u(x)$, satisfies
\[
\overline{N_0} \subset N.
\]
We claim that $N_0$ is positively invariant with respect to $\theta$. The set $N_0$ is contained in $V_i$ and, by ($\chi$-2), the vector field $Y$ has the form
\[
Y = \sum_{j\geq i} \chi_j Y_j
\]
on it. Then the claim follows from the fact that on $N_0$ all the vector fields $Y_j$ with $j\geq i$ are tangent to the foliation $\mathscr{L}^s_i$ from Lemma \ref{olem}.

Since $\theta^t(p)\rightarrow p_{\infty}$ for $t\rightarrow +\infty$, there exists $t_1\geq t_0$ such that $\theta^{t_1}(p)\in N_0$. By the continuity of $\theta^{t_1}$, there exists a neighborhood $U$ of $p$ in $\overline{V}\cap W^u(x)$ such that
\[
\theta^{t_1}(U) \subset N_0.
\]
By the positive invariance of $N_0$, $\theta^t(p')$ belongs to $N_0$ for every $p'\in U$ and every $t\geq t_1$. It follows that 
\[
\omega_{\theta}(p')\in \overline{N_0} \subset N
\]
for every $p'\in U$. This proves that the map $\omega_{\theta}$ is continuous at $p$.
\end{proof}
 
\begin{prop}
\label{strictpos}
The set $V$ is strictly positively invariant with respect to both $\phi$ and $\theta$, meaning that for every $p\in \overline{V}\cap W^u(x)$ and every $t>0$ we have:
\[
\phi^t(p)\in V \qquad \mbox{and } \qquad \theta^t(p)\in V.
\]
\end{prop}

\begin{proof}
By the continuity of the orbits, it is enough to show that for every $p$ in the relative boundary of $V$ in $W^u(x)$ there is a positive number $\epsilon$ such that
\[
\phi^t(p) \in V \qquad \mbox{and} \qquad \theta^t(p) \in V \qquad \forall t\in (0,\epsilon).
\]
Let $p$ be in the relative boundary of $V$ in $W^u(x)$. Then $\rho_i(p)=1$ for some $i\in \{0,\dots,k-1\}$. By ($\rho$-6) we have
\[
\frac{d}{dt} \rho_i \circ \phi^t (p) \Bigr|_{t=0} = - \rho_i(p) < 0.
\]
This implies that there is a positive number $\epsilon$ such that $\rho_i(\phi^t(p))<1$ for every $t\in (0,\epsilon)$, and hence $\phi^t(p)\in V_i \subset V$ for every $t\in (0,\epsilon)$.

Now let $i$ be the largest index in $\{0,\dots,k-1\}$ such that $\rho_i(p)=1$. By Lemma \ref{olem} we have
\[
\frac{d}{dt} \rho_i \circ \phi_Y^t (p) \Bigr|_{t=0} = - \chi_i(p) \rho_i (p).
\]
Being on the relative boundary of $V$ and not on the relative boundary of $V_j$ for any $j>i$, the point $p$ does not belong to $\overline{V_j}$ for any $j>i$. Then ($\chi$-3) implies that $\chi_i(p)>0$ and hence the above derivative is strictly negative. We conclude   that there is a positive number $\epsilon$ such that $\rho_i(\theta^t(p))<1$ for every $t\in (0,\epsilon)$, and hence $\theta^t(p)\in V_i \subset V$ for every $t\in (0,\epsilon)$.
\end{proof}

\begin{rem}
\label{estensione}
The above two propositions and the continuity of the local flow $\theta$ imply that $\overline{V}\cap W^u(x)$ has an open neighborhood $V'$ in $\dot{W}^u(x)$ such that the forward orbit of every $p\in V'$ is defined for every $t\in [0,+\infty)$. Indeed, by Proposition \ref{strictpos} the image of the boundary $\partial V$ of $V$ in $W^u(x)$ under $\theta^1$ is a compact set that is contained in the open set $V$. By continuity,  there is an open neighborhood $N$ of $\partial V$ in $W^u(x)$ such that $\theta^1(N)$ is contained in $V$ and our claim holds for $V'=V\cap N$ thanks to the first statement in Proposition \ref{converge}.
\end{rem}

\begin{prop}
\label{negcompl}
The complement $V^c$ of $V$ in $\dot{W}^u(x)$ is negatively invariant with respect to $\theta$. Morover, for every $p\in V$ there exists $t<0$ such that $\theta^t(p)\in V^c$.
\end{prop}

\begin{proof}
The negative invariance of $V^c$ follows from the positive invariance of $V$, which follows from Proposition \ref{strictpos}. In order to prove the second statement, we argue by contradiction and assume that there exists a point $p\in V$ whose backward orbit
\[
\gamma: (T,0] \rightarrow \dot{W}^u(x), \qquad \gamma(t):= \theta^t(p),
\]
is contained in $V$. Here $T\in [-\infty,0)$ is the infimum of the maximal interval of existence. Since $V$ is the finite union of the sets $V_i$, there exists $i\in \{0,\dots,k-1\}$ and a sequence $(t_h)\subset (T,0]$ with $t_h\rightarrow T$ such that $\gamma(t_h)\in V_i$ for every $h$. We may also assume that $i$ is maximal with this property, meaning that there exists $T_0\in (T,0]$ such that
\begin{equation}
\label{neg1}
\rho_j(\gamma(t)) \geq 1 \qquad \forall j>i \quad \forall t\in (T,T_0].
\end{equation}
By Lemma \ref{olem} and since $(t_h)$ tends to $T$, we have
\begin{equation}
\label{neg2}
\gamma(t) \in V_i \qquad \forall t\in (T,0],
\end{equation}
and
\[
\frac{d}{dt} \rho_i \circ \gamma (t) \leq 0 \qquad \forall t\in (T,0],
\]
from which we deduce
\begin{equation}
\label{neg3}
\rho_i(\gamma(t)) \geq \rho_i(\gamma(0)) \qquad \forall t\in (T,0].
\end{equation}
If $j<i$, the fact that $\chi_j$ vanishes on $V_i$ and Lemma \ref{olem} imply
\begin{equation}
\label{neg4}
\rho_j(\gamma(t)) = \rho_j(\gamma(0))\in (0,+\infty] \qquad \forall t\in (T,0].
\end{equation}
From (\ref{neg1}), (\ref{neg2}),  (\ref{neg3})  and (\ref{neg4}) we deduce that $\gamma((T,0])$ is contained in a compact subset of $\dot{W}^u(x)$. Therefore, $T=-\infty$. Using again (\ref{neg1}), (\ref{neg2}),  (\ref{neg3})  and (\ref{neg4}), we obtain that the $\alpha$-limit of $\gamma$ is a compact invariant set of $\theta$ which is contained in $\overline{V_i}$  and on which all the functions $\rho_j$ have a positive infimum. But Proposition \ref{converge} tells us that $\theta$ does not have an invariant set with these properties, as the forward orbits of points in $\overline{V_i}$ converge to points in $\overline{W^u(x)} \setminus W^u(x)$. This contradiction concludes the proof.
\end{proof}

\section{Proof of the main theorem}
\label{finalsec}

We are finally ready to prove the main result of this paper. As in the previous section, $x$ denotes a  stationary point of $\phi$ of Morse index $k$. 

\begin{thm}
\label{finale}
There exists a continuous flow $\psi$ on $W^u(x)$ with the following properties:
\begin{enumerate}[(i)]
\item The only stationary point of $\psi$ is $x$, and $x$ has a neighborhood $U\subset W^u(x)$ such that $\psi^t(p)=\phi^t(p)$ for every $p\in U$ and $t\leq 0$.
\item The $\alpha$-limit of every $p\in W^u(x)$ is the stationary point $x$: 
\[
\lim_{t\rightarrow -\infty} \psi^t(p) = x \qquad \forall p\in W^u(x).
\]
\item For every $p\in W^u(x)\setminus \{x\}$ we have
\[
\lim_{t\rightarrow +\infty} \psi^t(p) = \omega_{\psi}(p),
\]
where $\omega_{\psi}(p)$ is some point in $\overline{W^u(x)}\setminus W^u(x)$.
\item The map $p\mapsto \omega_{\psi}(p)$ is continuous on $W^u(x)\setminus \{x\}$.
\end{enumerate}
\end{thm}

As immediate corollary we deduce the following result, in which $B^k$ denotes the open unit ball in $\R^k$.

\begin{cor}
\label{daldisco}
If $x$ is a stationary point of index $k$ of the Morse-Smale gradient-like smooth flow $\phi$, then there exists a continuous map
\[
\varphi: \overline{B^k} \rightarrow \overline{W^u(x)}
\]
whose restriction to the interior maps $B^k$ homeomorphically onto $W^u(x)$. 
\end{cor}

The above statement is precisely Theorem \ref{introthm1} from the Introduction, with the further assumption that $\phi$ is smooth. The case of a flow of class $C^1$ follows from the smooth case by structural stability, as discussed in the Introduction.

\begin{proof}
We fix a smooth $(k-1)$-dimensional sphere $S$ in $U$ that is transversal to the flow, such as for instance $S := \partial D^u_x(r)$ for $r$ small enough, and we choose a homeomorphism
\[
g: \partial B^k \rightarrow S.
\]
Let $\chi: (0,1) \rightarrow \R$ be an orientation preserving homeomorphism, such as for instance
\[
\chi(r) = \tan \frac{2\pi r-\pi}{2} .
\]
Define $\varphi: \overline{B^k} \rightarrow \overline{W^u(x)}$ in polar coordinates by
\[
\varphi( r u) = \left\{ \begin{array}{ll} \psi^{\chi(r)} (g(u)) & \mbox{if } 0<r<1, \\ x & \mbox{if } r=0, \\ \omega_{\psi}(g(u)) & \mbox{if } r=1, \end{array} \right.
\]
where $u\in \partial B^k$. The map $\varphi$ is readily seen to have the required properties.
\end{proof}

The flow $\psi$ of Theorem \ref{finale} is constructed by ``juxtaposing'' the restriction of the flow $\phi$ to $W^u(x)$ and the local flow $\theta$ that we have constructed in the previous section. The important properties of these two objects are the following
\begin{enumerate}[($\phi$-1)]
\item $\phi$ is a continuous flow on $W^u(x)$.
\item The open set $V\subset W^u(x)$ is strictly positively invariant with respect to $\phi$.
\item For every $p\in W^u(x)\setminus \{x\}$ there exists $t>0$ such that $\phi^t(p)\in V$. 
\item For every $p\in W^u(x)$ the backward orbit $\phi^t(p)$ converges to $x$.
\end{enumerate}
\begin{enumerate}[($\theta$-1)]
\item $\theta$ is a positively complete continuous local flow on an open neighborhood of $\overline{V}\cap W^u(x)$.
\item The open set $V\subset W^u(x)$ is strictly positively invariant with respect to $\theta$.
\item There exists a continuous map $\omega_{\theta}: \overline{V}\cap W^u(x) \rightarrow \overline{W^u(x)} \setminus W^u(x)$ such that for every $p\in \overline{V}\cap W^u(x)$ the curve $\theta^t(p)$ tends to $\omega_{\theta}(p)$ for $t\rightarrow +\infty$.
\item For every $p\in V$ there exists $t<0$ such that $\theta^t(p)\in V^c:= W^u(x) \setminus V$.
\end{enumerate}
Properties ($\phi$-1) and ($\phi$-4) are clear. Property  ($\phi$-2) holds because the forward orbit of any $p\in W^u(x)$ converges a stationary point on the closure of $\overline{W^u(x)}$ and hence eventually enters $V$. Properties ($\phi$-2), ($\theta$-1), ($\theta$-2), ($\theta$-3) and ($\theta$-4) have been proven in Propositions \ref{converge}, \ref{strictpos} and \ref{negcompl} (see also Remark \ref{estensione}). 
Thanks to the conditions ($\phi$-1), ($\phi$-2), ($\theta$-1), ($\theta$-2) and ($\theta$-4), we can define a new flow
\[
\psi = \phi \#_V \theta
\]
on $W^u(x)$ that coincides with $\phi$ as long as we are outside of $V$ and inside $V$ switches to $\theta$. This new flow is defined by
\[
\psi^t(p) := \left\{ \begin{array}{ll} \theta^{(t-\tau(p))^+} \circ \phi^{t\wedge \tau(p)}(p) & \mbox{if } p\in V^c, \\ \phi^{-(t-\sigma(p))^-} \circ \theta^{t\vee \sigma(p)}(p) & \mbox{if } p\in \overline{V}, \end{array} \right.
\]
where $\tau$ and $\sigma$ denote the entrance times functions into $V$ of $\phi$ and $\theta$, respectively:
\[
\tau(p) := \inf \{ t\in \R \mid \phi^t(p) \in V \}, \qquad  \sigma(p) := \inf \{ t\in \R \mid \theta^t(p) \in V \}.
\]
Notice that the function $\tau$ is real valued on $W^u(x)\setminus \{x\}$ by ($\phi$-3).  Both $\tau$ and $\sigma$ are continuous as $\overline{\R}$-valued functions thanks to ($\phi$-2) and ($\theta$-2). See Appendix \ref{juxtsec} for more details about the juxtaposition of flows (under slightly more general assumptions) and for the proof that the above formula indeed defines a continuous flow on $W^u(x)$.

We just need to check that $\psi$ satisfies the conditions (i)-(iv) that are stated in Theorem \ref{finale}. The point $x$ is the only stationary point of $\psi$, because it is the only stationary point of $\phi$, it is not in $V$ and thanks to ($\theta$-4). Moreover, $\psi^t(p)=\phi^t(p)$ for every $t\leq 0$ and $p\in V^c$, which is a neighborhood of $x$. This proves (i).
By the same facts, together with ($\phi$-4), (ii) holds: the $\alpha$-limit of any $p\in W^u(x)$ under $\psi$ is $x$. If $p\in V$ then $\psi^t(p)$ converges to $\omega_{\theta}(p)$ for $t\rightarrow +\infty$, thanks to ($\theta$-3). If $p\in V^c\setminus \{x\}$, then $\psi^t(p)$ converges to $\omega_{\theta}(\phi^{\tau(p)})$, thanks to  ($\phi$-3) and ($\theta$-3). Therefore, $\psi$ satisfies conditions (iii) and (iv), where the map $\omega_{\psi}$ is  the composition
\[
\omega_{\psi}(p) := \omega_{\theta}\bigl( \phi^{\tau(p)^+}(p) \bigr),
\]
which is clearly continuous on $W^u(x)\setminus \{x\}$.

\begin{rem}
As any smooth flow, $\phi$ is the flow of a smooth vector field $X$. The local flow $\theta$ is also defined by integrating a vector field $Y$, but of lower regularity. Instead of juxtaposing the two flows, one could have tried to obtain $\psi$ as the flow of a unique vector field on $W^u(x)$, produced by patching $X$ and $Y$ by some smooth partition of unity, so that $X$ acts near the stationary point $x$ while $Y$ acts outside of a compact neighborhood of $x$. A difficulty with this approach is that it is not clear whether the resulting vector field is uniquely integrable. Indeed, $X$ is uniquely integrable because it is smooth, while $Y$ is uniquely integrable because it is tangent to suitable foliations and is smooth on each of their leaves. By patching together $X$ and $Y$ both properties would be lost.
\end{rem}

\appendix

\section{A gradient flow of a Morse function with an unstable manifold that is not a cell}
\label{nonsmale}

\numberwithin{equation}{section}
\setcounter{equation}{0}

Throughout this appendix, we call a subset $W\subset M$ of a manifold $M$ a $k$-cell if it is the image of a homeomorphism from the open $k$-ball $B^k\subset \R^k$ that extends to a continuous map from the closure of $B^k$ into $\overline{W}\subset M$. Theorem \ref{introthm1} from the Introduction states that the unstable manifolds of stationary points of a Morse-Smale gradient-like flow on a closed manifold $M$ are cells. Here we wish to discuss what happens if we drop the Smale transversality assumption.

If the closed manifold $M$ is one-dimensional, then every Morse gradient-like flow is automatically Smale, just because the unstable manifolds of its stationary points are either points or open intervals. 
In dimension two, the negative gradient flow of a Morse function need not be Smale. The standard counterexample is given by the height function $f(x,y,z)=z$ on the torus of revolution $M\subset \R^3$ that is obtained by revolving the circle
\[
\{(0,y,z) \mid y^2 + (z-2)^2 = 1\}
\]
around the $y$-axis. This Morse function has two critical points $p=(0,0,1)$ and $q=(0,0,-1)$ of index one, and the negative gradient flow with respect to the metric induced by the ambient Euclidean metric of $\R^3$ is such that the intersection
\[
W^u(p) \cap W^s(q) = \{(x,0,z) \mid x^2 + z^2 = 1\} \setminus \{p,q\}
\]
is non-transverse. In dimension two, however, the unstable manifold of a stationary point of index $k$ of a Morse gradient-like flow is always a $k$-cell, regardless of the Smale condition. This is obivious for stationary points of index zero and one - in any dimension - and when $\dim M=2$ it can be proven also for stationary points of index two because their unstable manifolds are open balls with a nice boundary. If the closed manifold $M$ is two-dimensional, one could also prove that the Smale condition is equivalent to the fact that the cells given by the unstable manifolds of all stationary points of a Morse gradient-like flow form a $CW$-decomposition of $M$.

Starting from dimension three, it is possible to construct negative gradient flows of Morse functions having some critical point whose unstable manifold is not a cell. Here we exhibit an example in dimension three. Higher dimensional examples can be constructed in a similar fashion, or by lifting this three-dimensional example to some product manifold.

We start by considering a smooth Morse function $f: \R^3 \rightarrow \R$ with the following properties:
\begin{enumerate}[(a)]
\item The function $f$ is odd and symmetric with respect to the $z$-axis: 
\[
f(x,y,z) = \hat{f}\bigl(\sqrt{x^2+y^2},z\bigr),
\]
with $\hat{f}(r,-z)= - \hat{f}(r,z)$.
\item There exists $R>0$ such that $\hat{f}(r,z) = z$ for $z^2+r^2 \geq R^2$.
\item The function $z\mapsto \hat{f}(0,z)$ has four non-degenerate critical points on $\R$: $-b<-a<a<b$ which are a local maximum, a local minimum, a local maximum and a local minimum, respectively, with
\[
\hat{f}(0,a) = \alpha > \beta = \hat{f}(0,b) > 0.
\]
\item The function $f$ has four critical points on $\R^3$, namely $(0,0,-b)$, $(0,0,-a)$, $(0,0,a)$ and $(0,0,b)$ of Morse index one, zero, three and two, respectively.
\item The stable and unstable manifolds with respect to the negative gradient flow $\phi_0$ of $f$ induced by the Euclidean metric on $\R^3$ are the following subsets:
\begin{enumerate}[(i)]
\item $W^u_0(0,0,-b)$ is the half-line $\{(0,0,z) \mid z< -a\}$;
\item $W^s_0(0,0,-b)$ is the sphere $\Sigma$ of radius $b$ centered at the origin, minus the point $(0,0,b)$;
\item $W^u_0(0,0,-a)= \{(0,0,-a)\}$;
\item $W^s_0(0,0,-a)$ is the open ball of radius $b$ centered at the origin, minus the line-segment from $(0,0,a)$ to $(0,0,b)$.
\end{enumerate}
By the symmetry property (a), we then have: 
\[
\begin{split}
W^u_0(0,0,a) = - W^s_0(0,0,-a) \qquad W^s_0(0,0,a) = - W^u_0(0,0,-a), \\
W^u_0(0,0,b) = - W^s_0(0,0,-b) \qquad W^s_0(0,0,a) = - W^u_0(0,0,-a).
\end{split}
\]
\end{enumerate}

\begin{figure}[h]
\label{fig1}
\centerline{\scalebox{.3}{\includegraphics{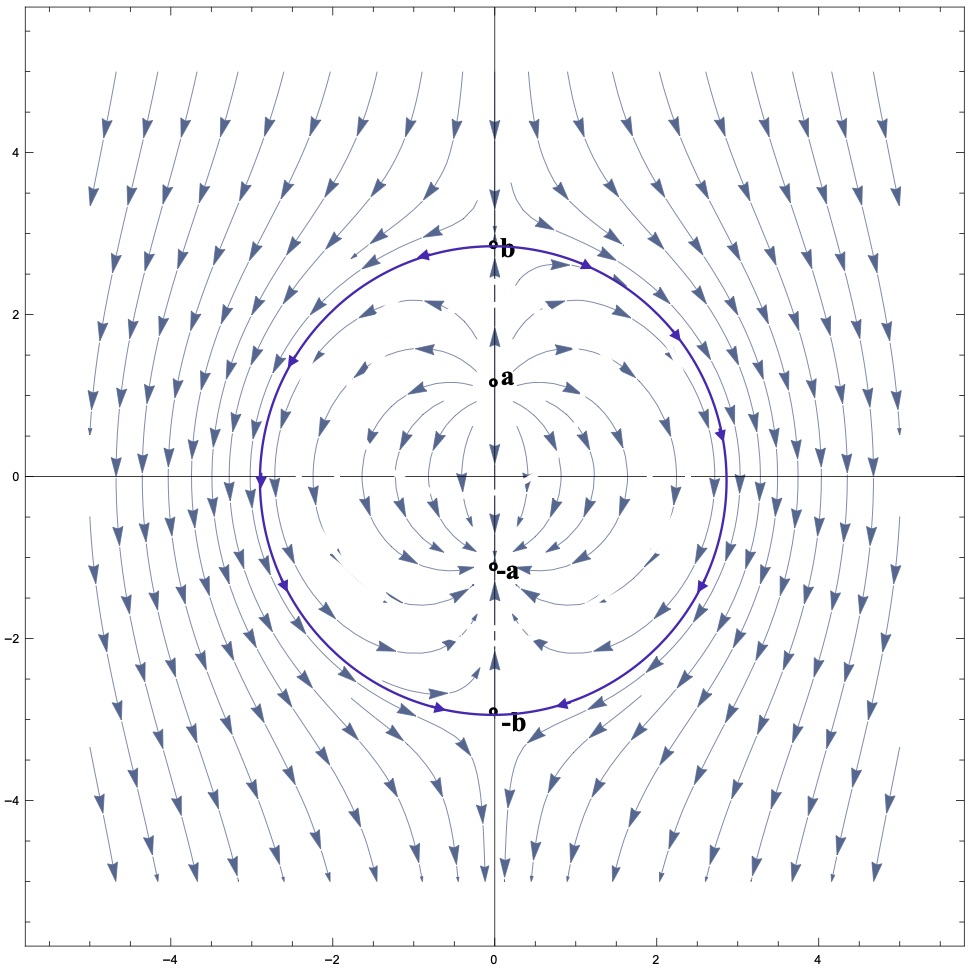}}}
\caption{Phase portrait of the flow $\phi_0$ in a plane containing the $z$-axis.}
\end{figure}

Figure 1 represents the flow lines of $\phi_0$ in a plane passing through the $z$-axis.  An explicit function that satisfies all these assumptions, except for (b), is defined by choosing
\[
\hat{f}(r,z) = z e^{\frac{8}{r^2+z^2+3}}.
\]
Indeed, $\hat{f}$ has the symmetry properties required in (a) and a simple computation shows that
\[
(\partial_r \hat{f}, \partial_z \hat{f}) = \rho F
\]
with
\[
\begin{split}
\rho(r,z) &:= \frac{1}{(r^2+z^2+3)^2} e^{\frac{8}{r^2+z^2+3}} > 0 , \\ F(r,z) &:= \bigl(-16 rz, (r^2+z^2+4z+3)(r^2+z^2-4z+3) \bigr).
\end{split}
\]
This implies that the critical points of $f$ are $(0,0,-3)$, $(0,0,-1)$, $(0,0,1)$ and $(0,0,3)$, at which we have
\[
f(0,0,3) = - f(0,0,-3) = e^{2/3}, \qquad f(0,0,1) = - f(0,0,-1) = e^2.
\] 
Moreover, the Hessian matrixes of $\hat{f}$ at $(0,3)$ and $(0,1)$ are easily computed to be
\[
\hat{f}'' (0,3) = \frac{e^{2/3}}{3} \left( \begin{array}{cc} -1 & 0 \\ 0 & 1 \end{array} \right), \qquad   
\hat{f}'' (0,1) = e^2 \left( \begin{array}{cc} -1 & 0 \\ 0 & -1 \end{array} \right).
\]
This implies that the critical points $(0,0,3)$ and $(0,0,1)$ have Morse index two and three, respectively. By symmetry, the critical points $(0,0,-3)$ and $(0,0,-1)$ have Morse index one and zero, respectively. This shows that $f$ satisfies (c) and (d) with $a=1$, $b=3$, $\alpha=e^2$ and $\beta=e^{2/3}$. The properties (i) and (iii) in (e) are easy to check. Properties (ii) and (iv) follow from the fact that the function
\[
\hat{g}(r,z) := -r + \frac{8r}{r^2+z^2-1} = r \frac{r^2+z^2-9}{r^2+z^2-1}
\]
satisfies 
\[
\partial_r \hat{f} \partial_r \hat{g} + \partial_z \hat{f} \partial_z \hat{g} = 0
\]
and hence the function 
\[
g(x,y,z) := g\bigl( \sqrt{x^2+y^2},z\bigr)
\]
is constant along the negative gradient flow lines of $f$. In particular, the sphere 
\[
\Sigma = \{r^2+z^2=9\} \subset g^{-1}(0)
\]
is invariant under $\phi_0$ and since its tangent space at $(0,0,3)$ is precisely the negative eigenspace of the Hessian of $f$,  (ii) holds and (iv) easily follows. The missing property (b) can be easily achieved by modifying $\hat{f}$ outside of a large ball.

The negative gradient flow $\phi_0$ of $f$ is not Smale because the two-dimensional unstable manifold of $(0,0,b)$ is nowhere transversal to the two-dimensional stable manifold of $(0,0,-b)$: their intersection is the two-dimensional submanifold $\Sigma\setminus \{(0,0,b),(0,0,-b)\}$. We shall now perturb the Euclidean metric so that the unstable manifold of $(0,0,b)$ with respect to the negative gradient flow $\phi$ of $f$ induced by the perturbed metric is not a 2-cell. For this purpose, consider the sequence of points $p_k$ on $\Sigma$ whose cylindrical coordinates are $r=3$, $\theta=1/k$ and $z=0$. The sequence $p_k$ converges to the point $p:=(1,0,0)\in \Sigma$. For every even natural number $k$, we can perturb the Euclidean negative gradient of $f$ in a small neighborhood of $\phi_0^1(p_k)$ by adding a small vector field that is normal to $\Sigma$ and inward pointing. If we denote by $\phi$ the flow of the perturbed vector field, we obtain that $\phi^t(p_k)$ converges to $(0,0,-a)$ for $t\rightarrow +\infty$, for every even $k$. The support of this perturbation can be chosen in such a way that the orbits of $p_k$ for $k$ odd are unaffected, and in particular still converge to $(0,0,-b)$ for $t\rightarrow +\infty$. The backwards orbits of all $p_k$'s are also unaffected, and hence they all belong to the unstable manifold $W^u(0,0,b)$ with respect to $\phi$.

By choosing the perturbation to be small, we obtain that $f$ is still a Lyapunov function for $\phi$. Together with the fact that the perturbation is supported away from the critical set of $f$, we obtain that $\phi$ is the negative gradient flow of $f$ with respect to some small perturbation of the Euclidean metric. We claim that the two-dimensional unstable manifold $W^u(0,0,b)$ is not a 2-cell.

Assume by contradiction that there is a continuous map
\[
h: \overline{B}^2 \rightarrow \overline{W^u(0,0,b)}
\]
whose restriction to $B^2$ is a homeomorphism onto $W^u(0,0,b)$. Consider the curves
\[
\gamma_k(t) := h^{-1}(\phi^t(p_k)), \qquad t\in \R.
\]
These curves are pairwise disjoint and converge to the common limit $h^{-1}(0,0,b)\in B^2$ for $t\rightarrow -\infty$. For $t\rightarrow +\infty$ each $\gamma_k(t)$ has a compact, connected and non-empty limit set $J_k$ that is contained in the circle $\partial B^2$. Since $h\circ \gamma_k(t)$ converges to $(0,0,-a)$ for $k$ even and to $(0,0,-b)$ for $k$ odd, we deduce that $f\circ h$ takes the value $-\alpha$ on the $J_k$'s with $k$ even and the value $-\beta$ on those with $k$ odd.
From the cyclic ordering of the curves $\gamma_k$ we obtain that the $J_k$'s are pairwise disjoint and cyclically ordered in the circle $\partial B^2$. By compactness of $\partial B^2$, we can find a $\xi$ in $\partial B^2$ that is the limit of a sequence $\xi_k$ with $\xi_k\in J_k$. The fact that $f\circ h(\xi_k)$ takes the value $-\alpha$ for $k$ even and $-\beta$ for $k$ odd violates the continuity of $f\circ h$ at $\xi$. This contradiction shows that a map $h$ with the above properties cannot exist, and hence $W^u(0,0,b)$ is not a 2-cell.

Thanks to property (b), any flow on a three-manifold can be locally modified near a non-stationary point by plugging in a copy of $\phi$ in a local chart.  If the original flow is a Morse gradient flow, so is the modified one. The modified flow has a  hyperbolic stationary point of Morse index two whose unstable manifold is not a 2-cell.

\section{Appendix: Juxtaposition of flows}
\label{juxtsec}

In this appendix we recall some terminology about (local) flows and we show how a flow can be juxtaposed to a positively complete local flow and produce a new flow. A {\em local flow} on the metric space $W$ is a continuous map
\[
\phi: \mathrm{dom}\, (\phi) \rightarrow W
\]
where $\mathrm{dom}\, (\phi)$ is an open neighborhood of $\{0\} \times W$ in $\R \times W$ such that for every $x\in W$ the set 
\[
I_x:= \set{t\in \R}{(t,x)\in \mathrm{dom}\, (\phi)}
\]
is an interval and the following group law holds: for every $(t,x)\in \mathrm{dom}\, (\phi)$ and every $s\in \R$ we have
\[
\bigl(s,\phi^t(x)\bigr) \in \mathrm{dom}\, (\phi) \qquad \Leftrightarrow \qquad (t+s,x)\in \mathrm{dom}\, (\phi),
\]
and in this case
\[
\phi^{t+s}(x) = \phi^s \bigl( \phi^t(x) \bigr).
\]
Here we are using the standard notation
\[
\phi^t(x) := \phi(t,x), 
\]
An open subset $V$ is said to be {\em positively invariant} if for every $x\in V$ and every $t\in I_x \cap [0,+\infty)$ we have $\phi^t(x)\in V$. If moreover for every $x\in \partial V$ and every $t\in I_x \cap (0,+\infty)$ we have $\phi^t(x)\in V$, the set $V$ is said to be {\em strictly positively invariant}.

The local flow $\phi$ is said to be {\em positively} (resp.\ {\em negatively}) {\em complete} if for every $x\in W$ the interval $I_x$ is unbounded from above (resp.\ from below). A local flow which is at the same time positively and negatively complete is a {\em flow}. 

Now let $A\subset W$ be positively (resp.\ negatively) invariant for the local flow $\phi$ on $W$. 
The local flow $\phi$ is said to be {\em positively} (resp.\ {\em negatively}) {\em complete relative to $A$} if for every $x\in W$ either the interval $I_x$ is unbounded from above (resp.\ from below) or there exists $t>0$ (resp.\ $t<0$) such that $\phi(t,x)\in A$. 

Assume that $V\subset W$ is a strictly positively invariant open subset of $W$ such that $\phi$ is positively complete relative to $V$ and negatively complete relative to $V^c$. Then the {\em entrance time} of $\phi$ in $V$, that is the function
\[
\tau: W \rightarrow \overline{\R}, \qquad \tau(x):= \inf \{ t\in \R \mid (t,x)\in \mathrm{dom}\, (\phi) \mbox{ and } \phi^t(x)\in V\},
\]
is continuous. Indeed, its upper semicontinuity follows from the fact that it is the first entrance time into an open set relative to which the flow is positively complete. By the strict positive invariance of $V$, $\tau$ is also the first entrance time into the closed set $\overline{V}$, and hence it is lower semicontinuous, by the negative completeness relative to $V^c$. Actually, 
\[
\tau^{-1}(\{0\})=\partial V, \quad \tau^{-1}([-\infty,0))=V, \quad \tau^{-1}((0,+\infty])=\overline{V}^c.
\]

Now let $W$ be a metric space and $V\subset W$ be an open subset. Let $\phi$ be a flow on $W$ and let $\theta$ be a positively complete local flow on an open neighborhood of $\overline{V}$. We assume that $V$ is strictly positively invariant for both $\phi$ and $\theta$. We also assume that $\theta$ is negatively complete with respect to $V^c$. Under these assumptions, we can define a new flow on $W$ which coincides with $\phi$ as long as we are outside $V$ and inside $V$ switches to $\theta$. In order to give a formal definition, we denote by $\tau$ the entrance time of $\phi$ in $V$ and by $\sigma$ be the entrance time of $\theta$ in $V$. By what we have observed above, the functions $\tau$ and $\sigma$ are continuous. We define a new map
\[
\psi: \R \times W \rightarrow W
\]
by
\begin{equation}
\label{juxtflow}
\psi^t(x) := \left\{ \begin{array}{ll} \theta^{(t-\tau(x))^+} \circ \phi^{t\wedge \tau(x)}(x) & \mbox{if } x\in V^c, \\ \phi^{-(t-\sigma(x))^-} \circ \theta^{t\vee \sigma(x)}(x) & \mbox{if } x\in \overline{V}. \end{array} \right.
\end{equation}
Here we are using the notation
\[
a^+ := \max\{a,0\}, \quad a^-:= - \min\{a,0\}, \quad a\vee b := \max\{a,b\}, \quad a \wedge b := \min\{a,b\}.
\]
Notice that if $x$ belongs to $V^c \cap \overline{V} = \partial V$, then $\tau(x) = \sigma(x) = 0$ and hence the two expressions
\[
 \theta^{(t-\tau(x))^+} \circ \phi^{t\wedge \tau(x)}(x) = \theta^{t^+} \circ \phi^{-t^-}(x)
 \]
 and
 \[
 \phi^{-(t-\sigma(x))^-} \circ \theta^{t\vee \sigma(x)}(x) = \phi^{-t^-} \circ \theta^{t^+} (x)
 \]
 define the same point in $W$. Therefore, the map $\psi$ is well defined. By the continuity of $\phi$, $\theta$, $\tau$ and $\sigma$, we deduce that $\psi$ is continuous on the closed sets $\R \times V^c$ and $\R \times \overline{V}$, and hence on their union $\R \times W$. This shows that the map $\psi$ is continuous.
 
 Clearly, $\psi^0$ is the identity on $W$. Checking the group law
 \[
 \psi^{s+t}(x) = \psi^s ( \psi^t(x) ) \qquad \forall x\in W, \; \forall s,t\in \R,
 \]
 consists in an elementary case distinction. We just treat the case $x\in V^c$ and leave the completely analogous case $x\in V$ to the patient reader. 
 
 Since $x$ is in $V^c$, we have
 \[
 \psi^{s+t}(x) = \theta^{(s+t-\tau(x))^+} \circ \phi^{(s+t)\wedge \tau(x)} (x) \qquad \mbox{and} \qquad \psi^t(x) = \theta^{(t-\tau(x))^+} \circ \phi^{t \wedge \tau(x)} (x).
 \]
 We distinguish the cases $t< \tau(x)$ and $t\geq \tau(x)$. In the case $t< \tau(x)$, the point $\psi^t(x)$ coincides with $\phi^t(x)$ and belongs to $V^c$. Moreover, 
 \[
 \tau(\psi^t(x)) = \tau(\phi^t(x)) = \tau(x)  - t,
 \]
 and hence
 \[
 \begin{split}
 \psi^s(\psi^t(x)) &= \theta^{(s-\tau(\psi^t(x)))^+} \circ \phi^{s\wedge \tau(\psi^t(x))}(\psi^t(x)) = \theta^{(s+t-\tau(x))^+} \circ \phi^{s\wedge (\tau(x)-t)} \circ \phi^{t} (x) \\ &= \theta^{(s+t-\tau(x))^+} \circ \phi^{(s+t)\wedge \tau(x)} (x) = \psi^{s+t}(x),
 \end{split}
 \]
 so the group law holds. In the case $t\geq \tau(x)$, the point $\psi^t(x)$ coincides with $\theta^{t-\tau(x)}\circ \phi^{\tau(x)} (x)$ and belongs to $\overline{V}$. Moreover,
 \[
 \sigma(\psi^t(x)) = \sigma\bigl( \theta^{t-\tau(x)}\circ \phi^{\tau(x)} (x) \bigr) = \tau(x) - t,
 \]
and hence
\[
\psi^s(\psi^t(x)) = \phi^{-(s-\sigma(\psi^t(x)))^-} \circ \theta^{s\vee \sigma(\psi^t(x))} (\psi^t(x)) = \phi^{-(s+t-\tau(x))^-} \circ \theta^{s\vee (\tau(x)-t)} \circ \theta^{t-\tau(x)} \circ \phi^{\tau(x)} (x).
\]
If $s+t-\tau(x)\geq 0$ then the last expression coincides with
\[
\theta^{s+t-\tau(x)}\circ \phi^{\tau(x)} (x) =  \theta^{(s+t-\tau(x))^+} \circ \phi^{(s+t)\wedge \tau(x)} (x)= \psi^{s+t}(x).
\]
If $s+t-\tau(x)< 0$ then it coincides with
\[
\phi^{s+t-\tau(x)} \circ \phi^{\tau(x)} = \phi^{s+t}(x) = \theta^{(s+t-\tau(x))^+} \circ \phi^{(s+t)\wedge \tau(x)} (x) = \psi^{s+t}(x).
\]
In both cases the group law is verified.

The new flow $\psi$ is also denoted by the symbol 
\[
\phi \#_V \theta: \R \times W \rightarrow W.
\]
We summarize the above result into the following statement.

\begin{prop}
Let $V$ be an open subset of the metric space $W$.
Let $\phi$ be a flow on $W$ and $\theta$ a positively complete local flow on an open neighborhood of $\overline{V}$. We assume that $V$ is strictly positively invariant for both $\phi$ and $\theta$, and that $\theta$ is negatively complete relative to $V^c$. Then the formula (\ref{juxtflow}) defines a flow $\psi=\phi \#_V \theta$ on $W$.
\end{prop}

%\bibliographystyle{amsalpha}
%\bibliography{../../biblio/nonlinear}

\providecommand{\bysame}{\leavevmode\hbox to3em{\hrulefill}\thinspace}
\providecommand{\MR}{\relax\ifhmode\unskip\space\fi MR }
\providecommand{\MRhref}[2]{%
  \href{http://www.ams.org/mathscinet-getitem?mr=#1}{#2}
}
\providecommand{\href}[2]{#2}

\end{document}